\documentclass[11pt]{article}
\usepackage{ifpdf}
\ifpdf
 \pdfoutput=1\pdfcompresslevel=9\pdfadjustspacing=1
 \usepackage[pdftex,bookmarks,a4paper,colorlinks]{hyperref}
 \usepackage{aeguill}
 \usepackage[pdftex]{graphicx}
 \usepackage{float}

\usepackage{tikz}
\usetikzlibrary{arrows,calc,decorations.text}
 \hypersetup{bookmarksnumbered,plainpages=false,hypertexnames=false}
 \hypersetup{pagecolor=blue,linkcolor=blue,citecolor=blue,urlcolor=blue}
 \hypersetup{pdfcreator=PDFLaTeX with TPAgregTeam class}
\else
 \usepackage[latin1]{inputenc}
 \usepackage[T1]{fontenc}
 \usepackage{url}
 \usepackage[dvips]{graphicx}
\fi

\usepackage[english]{babel}
\usepackage{amsmath,amsfonts,amsthm,amssymb,dsfont}
\usepackage{geometry}
\usepackage{enumerate}

\newcommand{\ABS}[1]{{\left| #1 \right|}} 
\newcommand{\PAR}[1]{{\left(#1\right)}} 
\newcommand{\SBRA}[1]{{\left[#1\right]}} 
\newcommand{\BRA}[1]{{\left\{#1\right\}}} 
\newcommand{\NRM}[1]{{\left\Vert #1\right\Vert}} 

\newcommand{\leqs}{\leq_\mathrm{sto.}}
\newcommand{\geqs}{\geq_\mathrm{sto.}}

\newcommand{\ind}{\mathds {1}}

\newcommand{\dE}{\mathbb{E}}

\newcommand{\dN}{\mathbb{N}}
\newcommand{\dP}{\mathbb{P}}

\newcommand{\dR}{\mathbb{R}}

\newcommand{\cC}{{\mathcal C }}

\newcommand{\cF}{{\mathcal F }}

\newcommand{\cL}{{\mathcal L }}

\newcommand{\cX}{{\mathcal X }}

\newcommand{\N}{{\scriptscriptstyle N}}

\newtheorem{thm}{Theorem}[section]

\newtheorem{prop}[thm]{Proposition}
\newtheorem{lem}[thm]{Lemma}
\newtheorem{defi}[thm]{Definition}
\newtheorem{rem}[thm]{Remark}
\newtheorem{ex}[thm]{Example}
\newtheorem{hyp}[thm]{Hypothesis}

\title{Long time behavior of  telegraph processes under convex potentials}

\author{%
  Joaquin~\textsc{Fontbona}, %
  H\'el\`ene~\textsc{Gu\'erin}, %
  Florent~\textsc{Malrieu}}

\begin{document}

\maketitle

\tableofcontents

\begin{abstract} 
We study the long-time behavior of   variants of the telegraph process with position-dependent jump-rates, which result in a monotone gradient-like  drift toward the origin. We compute their invariant laws and obtain, via   probabilistic couplings arguments,  some quantitative  estimates of the total variation distance  to equilibrium. Our techniques extend ideas previously developed for  a simplified piecewise deterministic Markov model 
of bacterial chemotaxis. 
\end{abstract}

\noindent\textbf{MSC Classification 2010:} 60F17, 60G50, 60J05, 60J60.

\section{Introduction}

\subsection{The model and main results}

Piecewise Deterministic 
Markov Process (PDMP)
have been extensively studied in the last two decades (see~\cite{Da,MR1283589,Jacob} 
for general  background) and  have recently received renewed attention,  motivated by 
their  natural application   in  areas such as biology~\cite{rousset,doumic}, communication 
networks~\cite{robertetal} or reliability of complex systems, to name a few. Understanding the ergodic properties of these models, in particular  the rate at which they stabilize towards equilibrium, has in turn increased the interest in  the long-time behavior of PDMPs.

In this paper we pursue the study of these questions on  PDMP models of bacterial chemotaxis, initiated in  \cite{othmer,EO} by means of analytic tools,  and  deepened in \cite{FGM1} and  \cite{MM} on simplified versions that can be seen as variants of Kac's classic ``telegraph process'' \cite{kac}. 

We consider the simple PDMP of kinetic type ${(Z_t)}_{t\geq 0}={((Y_t,W_t))}_{t\geq 0}$
with values in  $\dR\times\BRA{-1,+1}$ and infinitesimal generator 
\begin{equation}\label{eq:geninfsimple}
Lf(y,w)=w\partial_yf(y,w)
+\PAR{a(y)\ind_\BRA{yw\leq 0}+b(y)\ind_\BRA{yw>0}}(f(y,-w)-f(y,w)),
\end{equation}
where $a$ and $b$ are nonnegative functions in $\dR$. That is, the continuous 
component $Y$ evolves according to 
$\frac{d Y_t}{dt}= W_t$ and represents the position of a particle on the real line,  
whereas the discrete component $W$ represents the velocity of the particle 
and jumps between $+1$ and $-1$, with instantaneous state-dependent 
rate given by $a(y)$ (resp. $b(y)$) if the particle at position $y$ approaches 
(resp. goes away from) the origin. This process describes, in a naive way, the motion of flagellated bacteria as a 
sequence of linear "runs", the directions of which randomly change at rates 
that depend on the position of the bacterium. The emergence of macroscopical drift is expected when
the response mechanism favors longer runs in specific directions (reflecting the  propensity to move for instance towards 
a source of nutriments).  We refer the reader to \cite{rousset} for a scaling limit of  the processes introduced in \cite{othmer,EO}  that   leads to  simplified models like  \eqref{eq:geninfsimple}.

In the particular case where the jump rates 
are constants such that $b>a>0$, the  convergence to equilibrium of the process   \eqref{eq:geninfsimple}  has been investigated  in  a previous work~\cite{FGM1}, where  fully explicit and asymptotically sharp (in the natural diffusive scaling limit of the process) bounds  were obtained. We also  refer  the reader to  \cite{MM}  for the study of related models in the circle, relying on  a  spectral decomposition, and  to \cite{Monmarche}  for a general approach to some kinetic models including the above one,  based on functional inequalities.

In the present work we will consider position dependent jump-rates  which throughout will be  assumed to satisfy: 
\begin{hyp}\label{hyp:ab}
Function $b$ (resp. $a$) is  measurable, even, 
non decreasing   (resp. non increasing) on $[0,+\infty)$, bounded from below by 
$\underline{b}>0$ (resp. $\underline{a}>0$). Moreover we assume that 
 $b(y)>a(y)$ for all $y\neq0$.
\end{hyp}

\noindent In the sequel, $\bar b$ stands for $\sup_{y>0} b(y)\in [\underline b,\infty]$ 
and $\mathrm{sgn}:\dR\to \BRA{-1,+1}$  denotes de function
\[
\mathrm{sgn}(y)=\ind_\BRA{y\geq 0}-\ind_\BRA{y<0}.
\]

Let us denote by $\mu_t^{y,w}$ the law of $Z_t=(Y_t,W_t)$ when issued 
from $Z_0=(y,w)$. The following is a our main result:

\begin{thm}[Convergence to equilibrium]\label{th:no-reflection}
There exists $\kappa>0$, $K>0$,  
and $\lambda>0$ such that for any $y,\tilde y\in\dR$ and 
$w,\tilde w\in\BRA{-1,+1}$, 
 \begin{equation}   \label{eq:var-no-reflection}
 \NRM{\mu_t^{y,w}-\mu_t^{\tilde y,\tilde w}}_{\mathrm{TV}}%
 \leq K e^{\kappa \ABS{y}\vee \ABS{\tilde y}} e^{-\lambda t}.
\end{equation}
\end{thm}

  The constants above can be expressed in terms of the functions $a$ and $b$ 
following the lines of the proof. We  will  try to provide as explicit as possible bounds in each of its steps.  

The  proof of Theorem  \ref{th:no-reflection} relies on a probabilistic coupling argument,
 reminiscent of Meyn-Tweedie-Foster-Lyapunov techniques, see \cite{MT2,lindvall}.
Variants of this type of methods have  been developed in several previous works  on  specific instances of  PDMP \cite{BLMZ,tcp,azais,TAF}.
 The model 
under study in the present paper is harder to deal with, since the vector fields that 
drive the continuous part are not contractive.

 Our approach will be based  on extensions of  some ideas and methods  developed  in  \cite{FGM1}. However,  due to the non constant jump-rates  we will have to  work with explicit couplings of  
 jump-times,  for  two copies of the process found at different positions.  Moreover, we will need to make  use of some discrete time Markov  process embedded   in their trajectories (reminiscent of \cite{BCF}),  in order  to obtain  controls of the global coupling time.  These additional technicalities prevent us from getting estimates  as explicit   as in \cite{FGM1}.

\medskip

Before delving into the proof of Theorem  \ref{th:no-reflection}, we point out the explicit form of   the equilibrium of the process  $(Y,W)$  and its relation to  one dimensional  diffusion processes in a convex potential.

\begin{prop}[Invariant distribution]\label{Invariant distribution}
The invariant distribution of $(Y,W)$ on $\dR\times \BRA{-1,+1}$ is given by
\[
\mu(dy,dw)=\frac{1}{C_F}e^{-F(y)}dy 
 \otimes \frac{1}{2}(\delta_{-1}+\delta_{+1})(dw)
\]
where   $C_F:=\int_{\dR}\! e^{-F(y)}\,dy<\infty$ and $F$ is the convex function
\begin{equation}\label{eq:defF}
 y\in\dR\mapsto F(y)=\int_0^y\!\mathrm{sgn}(z)\PAR{b(z)-a(z)}\,dz. 
\end{equation}
The domain of the  Laplace transform  of $\mu$ is 
$(-\bar b+\underline a,\bar b-\underline a)\times \BRA{-1,+1}$. 
\end{prop}

\begin{ex}[Laplace and Gaussian equilibria]
If $a$ and $b$ are constant functions, then 
\[
\mu(dy,dw)=\frac{b-a}{2}e^{-(b-a)\ABS{y}}dy\otimes \frac{1}{2}(\delta_{-1}+\delta_{+1})(dw).
\]
If $a$ is a constant function and $b$ is the map $y\mapsto a+|y|$, then 
\[
\mu(dy,dw)=\frac{1}{\sqrt{2\pi}}e^{-y^2/2}dy\otimes \frac{1}{2}(\delta_{-1}+\delta_{+1})(dw).
\]
Figure \ref{fi:loivrai} compares in the latter case the empirical law of $Y_t$ to 
its invariant measure at increasing time instants.
\begin{figure}[!h]
\begin{center}
 \includegraphics[scale=0.2]{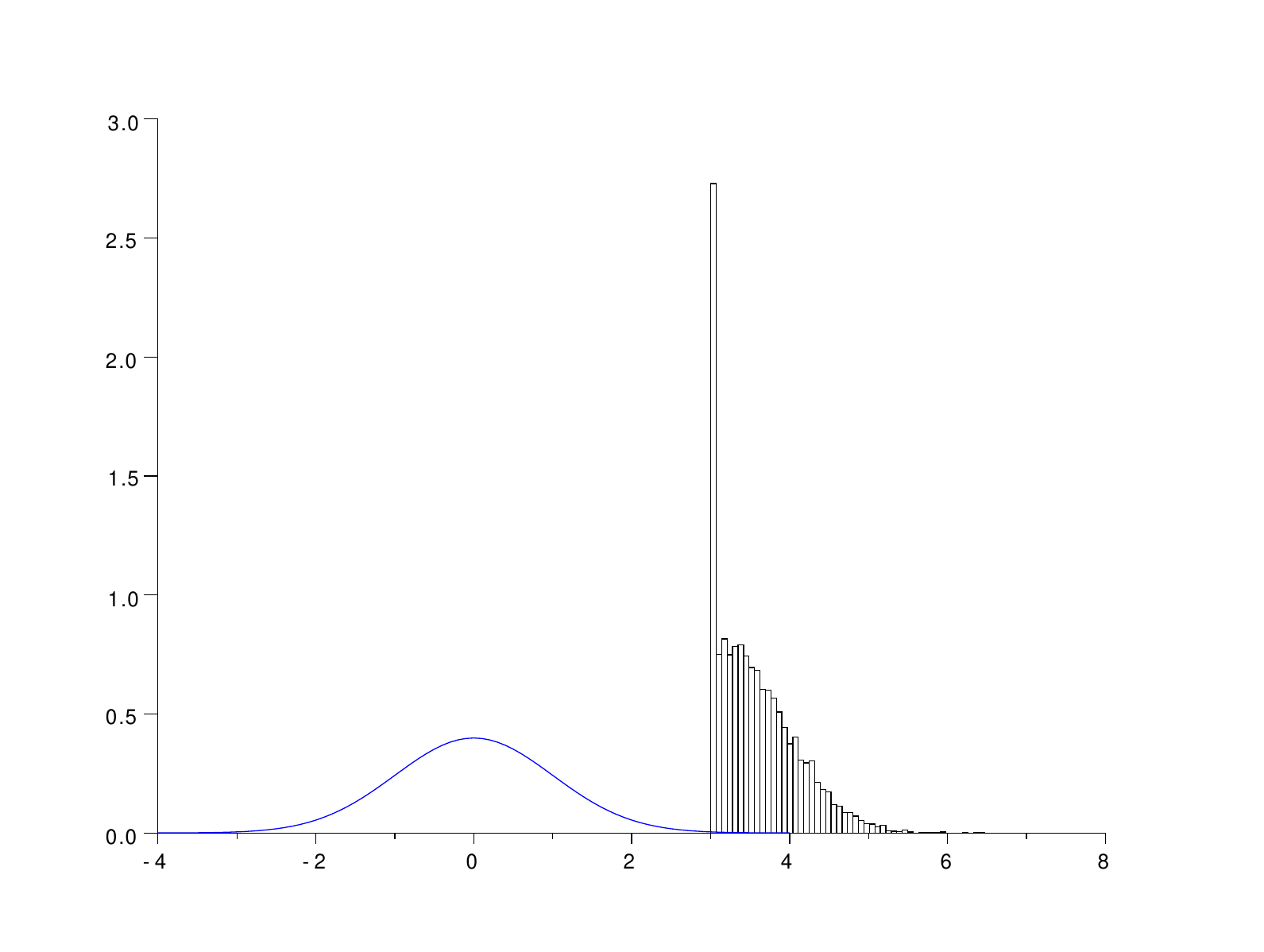}
 \includegraphics[scale=0.2]{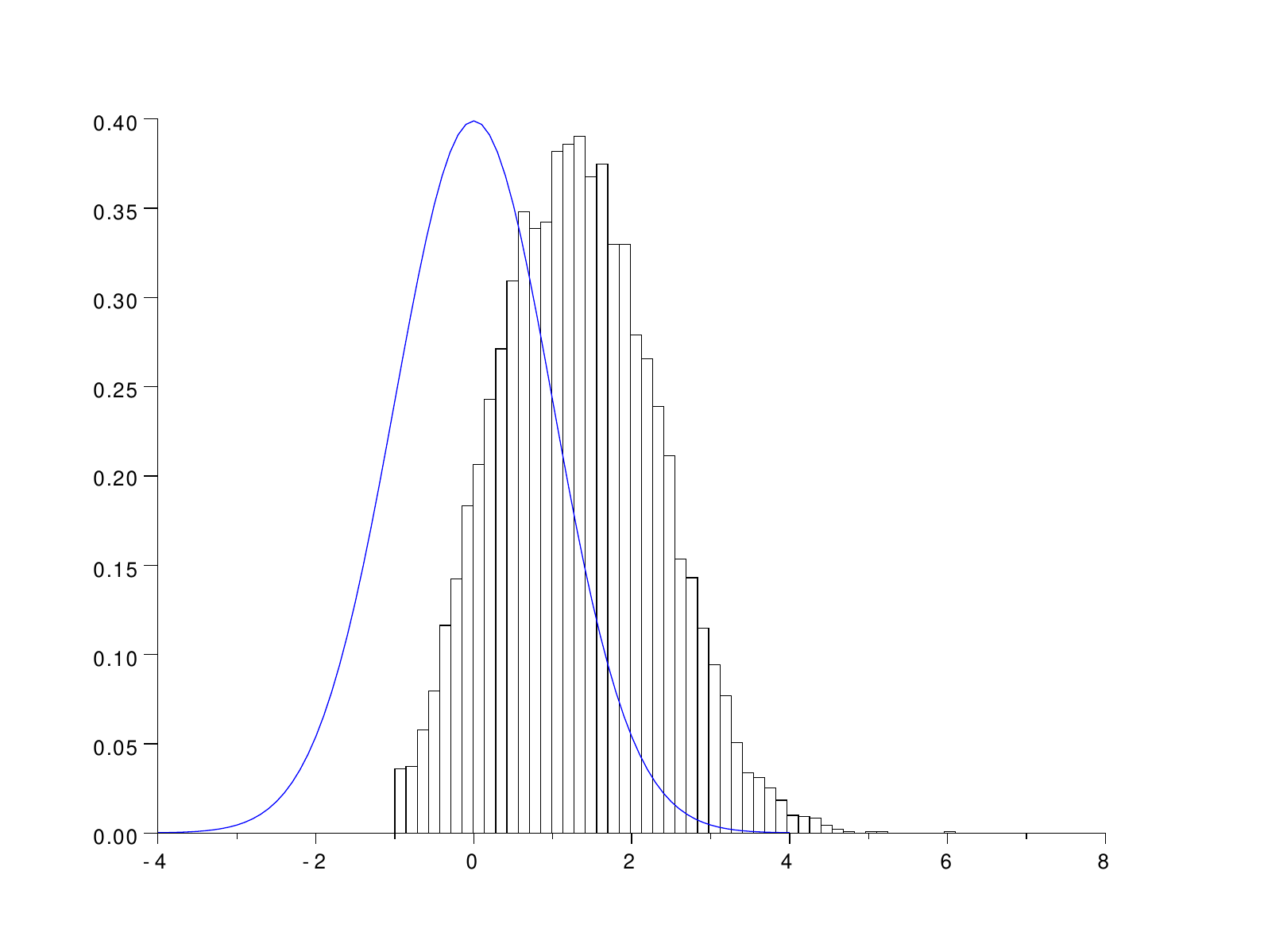}
 \includegraphics[scale=0.2]{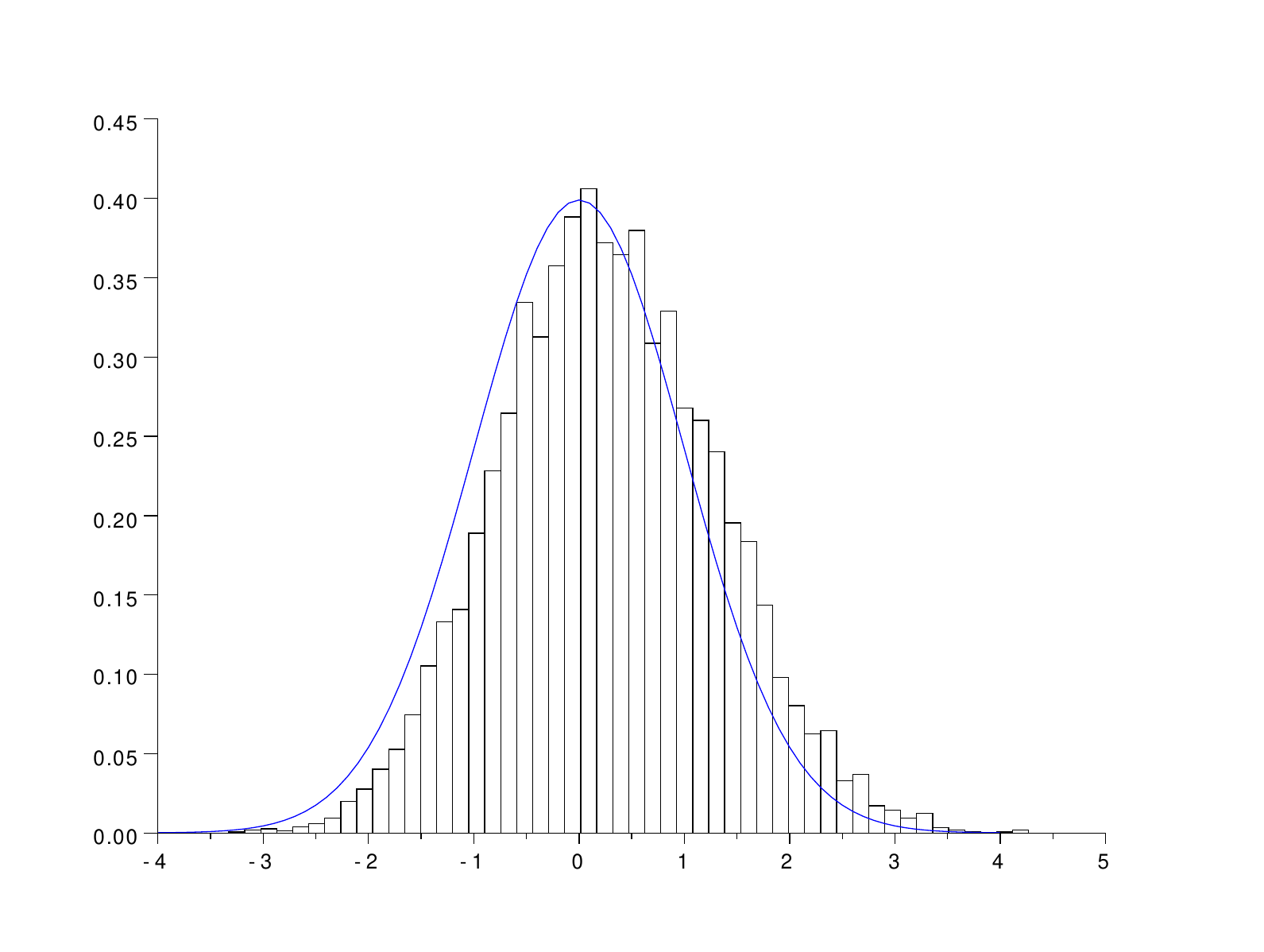}\\
 \includegraphics[scale=0.2]{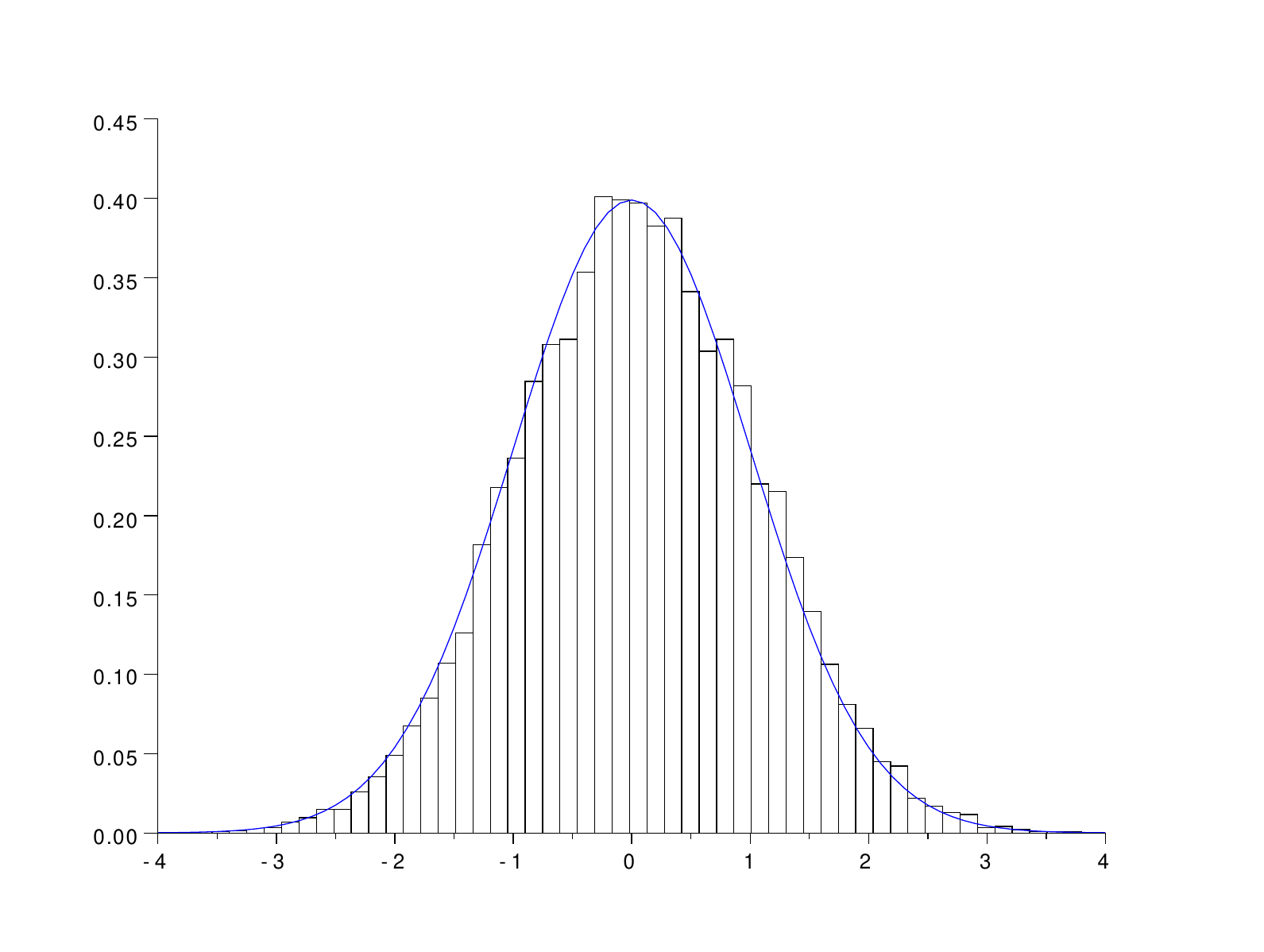}
 \includegraphics[scale=0.2]{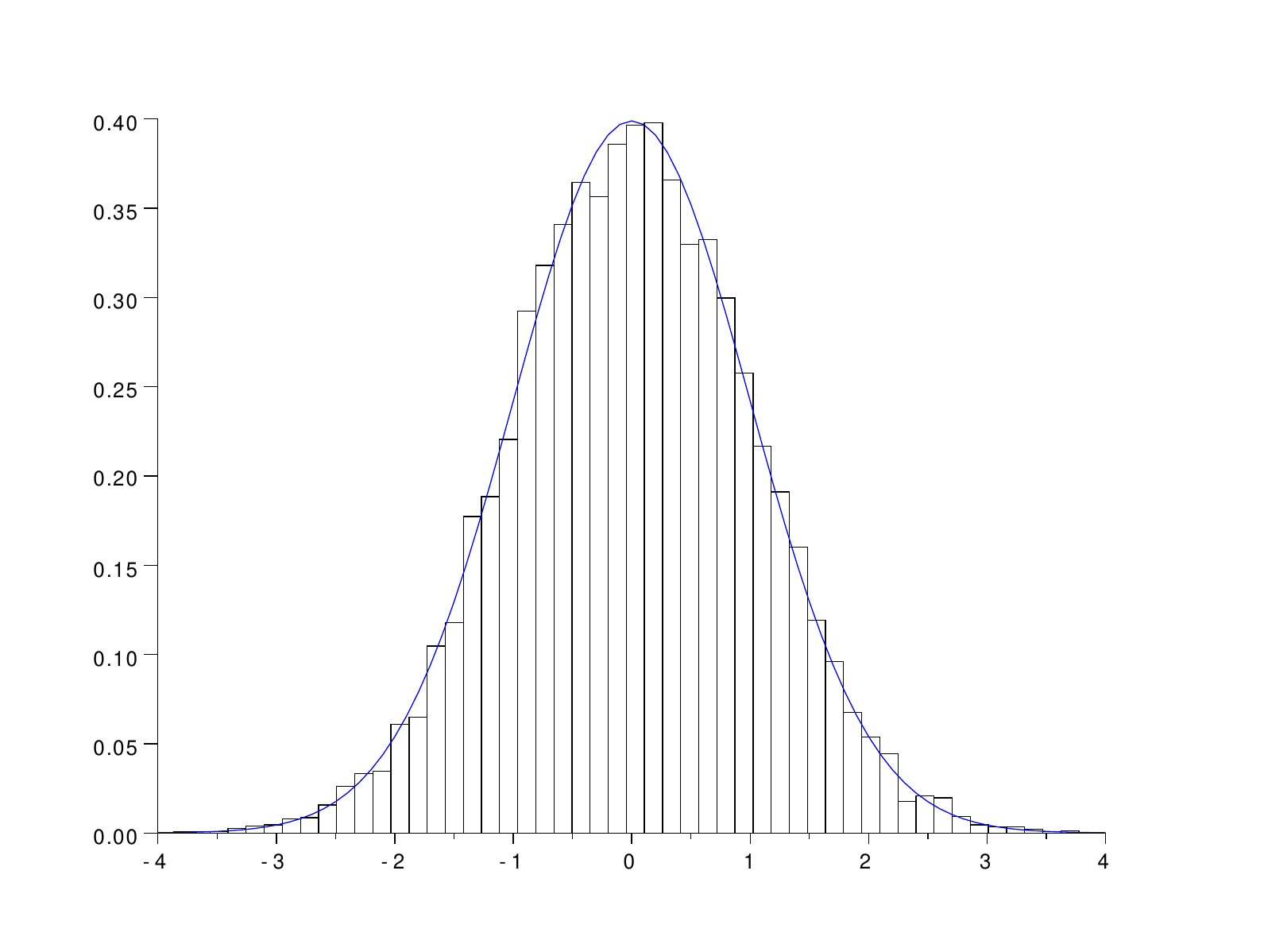}
 \includegraphics[scale=0.2]{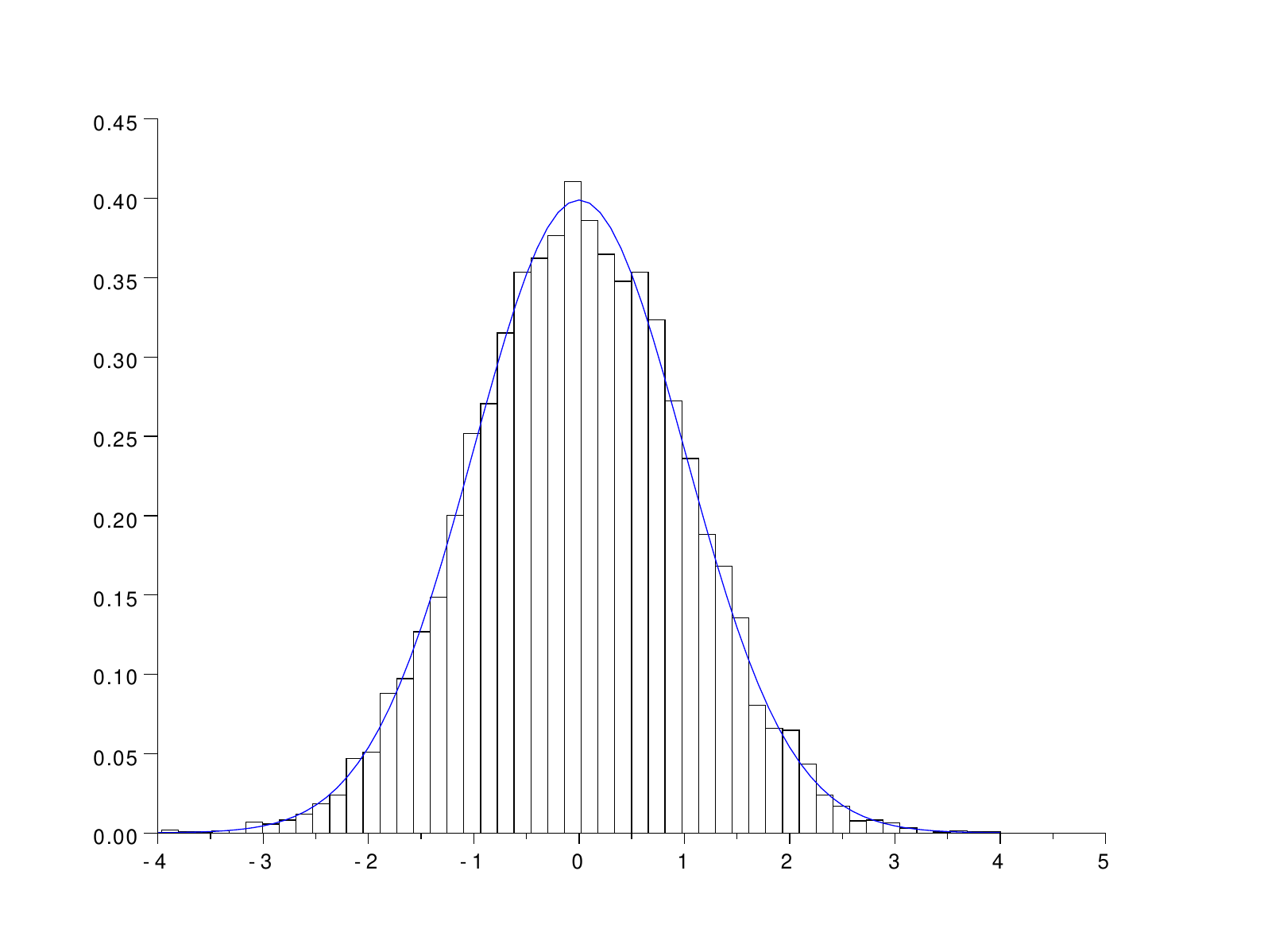}
  \caption{Empirical distribution of $Y_t$ starting at $(5,-1)$ for
  $t\in\BRA{2,6,10,14,18,22}$ with $a(y)=1$ and $b(y)=1+\ABS{y}$.}
   \label{fi:loivrai}
\end{center}
\end{figure}
\end{ex}

\begin{proof}[Proof of Proposition \ref{Invariant distribution}]
We first note that the constant $C_F$ is finite as soon as the functions $a$ and $b$ satisfy 
Hypothesis \ref{hyp:ab}, since $z\mapsto b(z)-a(z)$ is non decreasing and positive on 
$(0,\infty)$. Furthermore, 
\[
\lim_{y\to+\infty}\frac{1}{y}\int_0^y\! (b(z)-a(z))\,dz = \bar b-\underline a. 
\]
This ensures  that the Laplace transform of $\mu$ is finite on 
$(-\bar b+\underline a,\bar b-\underline a)\times \BRA{-1,+1}$ and  infinite on the complement. For any function $f\in\cC^1$ on $\dR\times \BRA{-1,+1}$ with compact support
one has, from the definition of $F$,
\begin{align*}
Lf(y,1)+Lf(y,-1)&=\partial_{y}(f(y,1)-f(y,-1))
-\mathrm{sgn}(y)\PAR{b(y)-a(y)}\PAR{f(y,1)-f(y,-1)}\\
&=\partial_{y}(f(y,1)-f(y,-1))-F'(x)\PAR{f(y,1)-f(y,-1)}.
\end{align*}
An integration by parts ensures that 
\[
\int\!\PAR{\partial_{y}f(y,1)-\partial_{y}f(y,-1)}e^{-F(y)}\,dy=
\int\!\PAR{f(y,1)-f(y,-1)} F'(y)e^{-F(y)}\,dy, 
\]
which yields
$\int\! Lf(y,w)\,\mu(dy,dw)=0$. 
In other words, $\mu$ is an invariant measure for $L$.
\end{proof}

The next result is independent of the previous and can be seen as a generalization of  the renormalisation of the telegraph process  studied by Kac \cite{kac} (see also
\cite{herrmann}, \cite{FGM1}). It shows that, under the suitable scaling,  process  \eqref{eq:geninfsimple} behaves like the   diffusion processes expected from  Proposition \ref{Invariant distribution}:

\begin{thm}[Diffusive scaling]\label{prop:difulimit}
For each  $N\geq 1$ let   $a_\N, b_N:\dR\to \dR_+$   be jump-rates satisfying Hypothesis~\ref{hyp:ab} 
such that
$y\mapsto a_N(y)+b_N(y)$ is of class ${\cal C}^1$ and
 \begin{enumerate}
\item[i)] $a_\N(0)+b_\N(0)\to \infty$,
\item[ii)] $b_\N-a_\N\to2 c_1$ and $\frac{a_N'+b_N'}{a_{N}+b_N}\to2c_2$ locally  uniformly  for some functions
 $c_1,c_2:\dR\to \dR$
\end{enumerate}
when $N \to \infty$.
Let $(Y^{(\N)}_t, W^{(\N)}_t)_{t\geq 0}$ denote the  
process driven by \eqref{eq:geninfsimple} with  $a=a_\N, b=b_\N$ and assume that $Y_0^{(\N)}\to \xi_0$  in law as $N\to \infty$.  Then, the sequence of processes 
\[
\left(\xi^{(\N)}_\cdot\right)_{N\geq 1}:=\left(Y^{(\N)}_{\tau_{\N}(\cdot)}\right)_{N\geq 1},
\]
with  $\tau_{\N}$  the  solution of $ 
\tau'_{\N}(t)=\frac{1}{2}(a_{\N}(Y^{(N)}_{\tau_{\N}(t)})+b_{\N}(Y^{(N)}_{\tau_{\N}(t)})), 
\, \tau_{\N}(0)=0,
$
weakly  converges  in $C([0,\infty), \dR)$ when  
$N\to \infty$  to the solution $\xi_t$ of the stochastic differential equation 
\begin{equation}\label{eq:bang}
d\xi_t=dB_t-\PAR{\mathrm{sgn}(\xi_t) c_1(\xi_t)+c_2(\xi_t)}\,dt  \, ,
\end{equation} 
where $(B_t)_{t\geq 0}$ is  a standard Brownian 
motion independent from $\xi_0$.
\end{thm}

\begin{rem}[Diffusion in a convex potential]
The drift term in~\eqref{eq:bang} is odd since $a_N$ and $b_N$ are even.
Notice from point  ii) above  that 
$\frac{a_N(y)+b_N(y)}{a_N(0)+b_N(0)}=e^{2\int_0^y c_2(s)ds + o(1)}$ 
uniformly on compact sets as $N\to \infty$. Hence 
$a_N(y)+b_N(y)\to \infty$ for all $\geq 0 $ by i), and $c_2=0$ if and 
only if $a_N(y)+b_N(y)\sim a_N(0)+b_N(0)$  for each $y\geq 0$. 
Thus, any  diffusion with generator of the form $\frac{1}{2}f''(y) -U'(y)f'(y)$ 
for an even convex potential $U$ can be obtained as a limit, taking for 
instance $a_N(y)=a_N(0)\longrightarrow\infty$ and $b_N(y)=a_N(0)+ 2U'(y)$.
\end{rem}

The remainder of the paper is organized as follows.  In the next subsection we briefly recall generalities on  the coupling approach to  long-time convergence in total variation distance. We also  define therein the reflected version of process   \eqref{eq:geninfsimple}, a detailed study of which is  crucial for proving  Theorem \ref{th:no-reflection}.  Section~\ref{sec:times} is devoted to the study of jump and hitting
times of the latter process. A coalescent coupling for it is then constructed  in Section~\ref{sec:coalescent-refl}  and the corresponding convergence estimates are established. The proof of Theorem~\ref{th:no-reflection} is achieved in Section~\ref{sec:non-reflected}. Finally, Theorem \ref{prop:difulimit}  is proved in 
Section~\ref{sec:diffus}.

\subsection{Preliminaries }\label{sec:prelim}
In the sequel we will use the notation 
$\overset{\cL}{=}$ meaning ``equal in law to'' and $E(\lambda)$ for an exponential 
random variable with mean $1/\lambda$.  

Recall that the total variation distance between two probability measures
$\eta$ and $\tilde \eta$ on a measurable space $\cX$ is given by
\begin{equation}\label{eq:def-var-tot}
\NRM{\eta-\tilde \eta}_{\mathrm{TV}}=
\inf\BRA{\dP(X\neq \tilde X)\, :\, X,\tilde X 
\mbox{ random variables  with }\cL(X)=\eta,\ \cL(\tilde X)=\tilde \eta}.
\end{equation}
If $\eta$ and $\tilde \eta$ are absolutely continuous 
with respect to a measure $\nu$ with respective densities $f$ and $\tilde f$ 
then 
\[
\NRM{\eta-\tilde\eta}_{\mathrm{TV}}=\frac{1}{2}\int\ABS{f-\tilde f}\,d\nu
=1-\int\! f\wedge \tilde f\,d\nu.
\]
See \cite{lindvall} for alternative definitions of this distance and its main properties. 
A pair of stochastic processes $(U_t, \tilde{U}_t)_{t\geq 0}$ constructed in the same probability space,
 for which an almost surely finite random time  $T$ satisfying  
$U_{t+T}=\tilde{U}_{t+T}$ for any $t\geq 0$ exists, is called a coalescent coupling. The random variable 
\[
T_*=\inf\BRA{t\geq 0: U_{t+s}=\tilde U_{t+s}\, \forall s\geq 0}
\]
is then called the {\it coupling time}. It follows in this case that
\[
\NRM{\cL(U_t)-\cL(\tilde{U}_t)}_\mathrm{TV}\leq \dP (T_*>t).
\]
A helpful notion in obtaining an  effective control of the distance is 
\emph{stochastic domination} (see~\cite{lindvall} for a complete 
introduction).

\begin{defi}[Stochastic domination]
 Let $S$ and $T$ be two real random variables with
 respective cumulative distribution functions $F$ and $G$. We
 say that $S$ is stochastically smaller than $T$ and we write 
 $S\leqs T$, if $F(t)\geq G(t)$ for any $t\in\dR$. 
\end{defi}

In particular, for a couple $(U_t, \tilde{U}_t)$ as above,  Chernoff's inequality  yields
\begin{equation}\label{coupleineg}
\NRM{\cL(U_t)-\cL(\tilde{U}_t)}_\mathrm{TV}\leq \dP (T>t) \leq  
\dE\PAR{e^{\lambda T}} e^{ -\lambda t}
\end{equation}
for any non-negative random variable $T$ such that 
$T_*\leqs T$, and any $\lambda\geq 0$ in the domain 
of the Laplace transform   $\lambda \mapsto \dE\PAR{e^{\lambda T}}$ of $T$.  

We will use these ideas to obtain  the exponential convergence to equilibrium
for $(Y,W)$   in Theorem \ref{th:no-reflection},  and in Theorem 
\ref{th:conv-reflection} below  for its reflected  version  $(X,V)$ which 
we now introduce. The Markov process ${((X_t,V_t))}_{t\geq 0}$ 
is defined by its infinitesimal generator:
\begin{equation}\label{eq:giabs}
Af(x,v)=v\partial_xf(x,v)+\PAR{a(x)\ind_\BRA{v=-1}+b(x)\ind_\BRA{v=1}
+\frac{\ind_\BRA{x= 0} }{\ind_\BRA{x> 0}}}(f(x,-v)-f(x,v)),
\end{equation}
where the maps $a$ and $b$ satisfy Hypothesis~\ref{hyp:ab}. The term 
$\ind_\BRA{x= 0}( \ind_\BRA{x> 0} )^{-1}$ means that $V$ flips from $-1$ to $+1$ 
as soon as $X$ hits zero. In other words, $X$ is reflected at zero.

\begin{rem}\label{rem:ru}
 Given a path ${((Y_t,W_t))}_{t\geq 0}$ driven by
\eqref{eq:geninfsimple},  a path of
${((X_t,V_t))}_{t\geq 0}$  can be constructed taking
 \[
X_t=\ABS{Y_t}, \quad V_0=\mathrm{sgn }(Y_0)W_0
\]
and  defining the set of jump times of $V$ to be 
\[
\BRA{t>0\,:\, \Delta V_t \neq 0}=%
\BRA{t>0\,:\, \Delta W_t \neq 0}\cup \BRA{t>0\,:\,Y_t=0}.
\]
Since $W$ does not jump with positive probability when
$Y$ hits the origin,  one can also construct a path of 
${((Y_t,W_t))}_{t\geq 0}$  from an initial value
$y\in\dR$ and a path ${((X_t,V_t))}_{t\geq 0}$ driven by
\eqref{eq:giabs}: set  $\sigma_0=0$ and   ${(\sigma_i)}_{i\geq 1}$ for the successive
hitting times   of the origin and
\[
(Y_t,W_t)=(-1)^{i}\mathrm{sgn}(y) (X_t,V_t)   \quad\text{if }t\in [\sigma_{i},\sigma_{i+1}).
\]
\end{rem}
Let us state our results about the long time behavior of $(X,V)$. 

\begin{thm}\label{th:conv-reflection}
 The invariant measure of $(X,V)$ is the product
 measure on $\dR_+\times\BRA{-1,+1}$ given by
\[
\nu(dx,dv)=\frac{2}{C_F}e^{-F(x)}\,dx\otimes  \frac{1}{2}(\delta_{-1}+\delta_{+1})(dv).
\]
where $F$ and $C_F$ are given by \eqref{eq:defF}.  Moreover, denoting by 
$\nu_t^{x,v}$ the law of $(X_t,V_t)$ when $X_0=x$ and $V_0=v$, 
there exists $\lambda >0$, $K>0$ and $c>0$ such that, for any $x,\tilde x\geq 0$ 
and $v,\tilde v\in\BRA{-1,+1}$,
\begin{equation}\label{eq:up-reflection}
\NRM{\nu_t^{x,v}-\nu_t^{\tilde x,\tilde v}}_{\mathrm{TV}}\leq %
Ke^{-\lambda t} e^{c(x\vee \tilde x)}.
 \end{equation}
\end{thm}

Following the lines of the proof, the constants $\lambda$, $K$ and $c$ can be 
expressed in terms of the jump rate functions $a$ and $b$.  Let us summarize some important random times involved in the proof of 
Theorem~\ref{th:conv-reflection} (all related to the reflected process): 
\begin{itemize}
 \item $T_{(x,v)}$ stands for the first jump time starting at $(x,v)$ and $\varphi_{(x,v)}$ stands for its 
 Laplace transform,
 \item $Z_{(x,v)}$ stands for the first hitting time of $(0,+1)$ starting at $(x,v)$, 
 \item $T_c$ stands for the first crossing time of the continuous components of two paths, 
 \item $T_*$ stands for the coupling time (i.e. from  $T_*$ on the two paths are equal forever).  
\end{itemize}

Roughly, we will let evolve both paths until the first crossing time $T_c$ (which is stochastically controlled 
by hitting times of 0) and then  couple the whole processes by using explicit couplings of the jump-times. Notice that it is not obvious to deduce 
Theorem~\ref{th:no-reflection} from Theorem~\ref{th:conv-reflection}.

\section{Basic properties of the reflected process}\label{sec:times}

\subsection{Distribution of the jump times}\label{sec:lawJT}

Let us denote by $T_\PAR{x,v}$ the first jump time of the stochastic process $(X,V)$ 
starting from $(X_0,V_0)=(x,v)$ with infinitesimal generator defined by \eqref{eq:giabs}. 
This random time satisfies
\[
T_\PAR{x,v}=\inf\BRA{t\geq 0: \int^{t}_0 c(X_s)\geq E},
\]
where $E$ is an exponential variable with unit mean and $c$ stands for the function $b$ 
when $v=+1$ and, when $v=-1$, 
\begin{eqnarray*}
c(x)=
\begin{cases}
a(x)& \text{ for } x>0,\\
+\infty& \text{ for } x\leq 0.
 \end{cases}
\end{eqnarray*}
 
\noindent The process $(X,V)$ being deterministic between jump times, we have
\[
T_\PAR{x,v}=\inf\BRA{t\geq 0: \int^{t}_0\!c(x+vs)\geq E} . 
\]
Consequently, if $B$ is the primitive of $b$ with $B(0)=0$, we have
\begin{equation}\label{eq:T+}
T_\PAR{x,+1}=B^{-1}\PAR{E+B(x)} -x,
\end{equation}
and if $A$ is the primitive of $a$ with $A(0)=0$, 
\begin{equation}\label{eq:T-}
T_\PAR{x,-1}=
\begin{cases}
x-A^{-1}\PAR{A(x)-E} &\text{if }E<A(x),\\
x&\text{otherwise.}
\end{cases}
\end{equation}
The functions $A^{-1}$ and $B^{-1}$ are well defined since $a$ and $b$ are positive functions.

\begin{lem}[Law of jump times]\label{lem:law-time}
Let $x\in\dR_+$. The random variable  $T_\PAR{x,+1}$ is absolutely continuous with density  given by
\[
t\mapsto b(t+x)e^{-\PAR{B(t+x)-B(x)}}\ind_{(0,\infty)}(t).
\]
The random variable $T_\PAR{x,-1}$ is a mixture of an absolutely  continuous random variable and the
constant variable $x$. Its distribution is
\begin{align*}
e^{-A(x)}\delta_{x}+a(x-t)e^{-\PAR{A(x)-A(x-t)}}\ind_{[0,x]}(t) dt,
\end{align*}
where $\delta_x$ denotes the Dirac mass at $x$ and $dt$ the Lebesgue measure on $\dR$.
\end{lem}

\begin{proof}
We notice that $T_\PAR{x,+1}$ is almost surely finite since $\int^{\infty}_0 b(x+s)ds=+\infty$. 
Let $E$ be an exponential variable with unit mean and $t\geq 0$, then
\[
\dP(T_\PAR{x,+1}>t)=\dP\PAR{\int^{t}_0 b(x+s)ds<E}=e^{-\int^{t}_0 b(x+s)ds}.
\]
We obtain the density of $T_{(x,+1)}$ by derivation. The distribution of $T_\PAR{x,-1}$ is similarly obtained noting that $\dP(E>A(x))=e^{-A(x)}$.
%
\end{proof}

\begin{lem}[Laplace transform of jump times]\label{lem:Laplace}
Let $x\geq 0$ be fixed. The Laplace transform of $ T_\PAR{x,+1}$ is finite 
on $(-\infty,\overline b)$. Furthermore, if $\lambda<b(x)$,
\[
\dE\SBRA{e^{\lambda T_\PAR{x,+1}}}\leq{b(x)\over b(x)-\lambda}.
\]
\end{lem}

\begin{proof}
Let $x\geq 0$ be fixed. Since $b$ is non decreasing, we first we notice that 
$T_\PAR{x,+1}\leq_\mathrm{sto.}E/b(x)$, where $E$ is an exponential 
variable with mean $1$. Then its Laplace transform is at least defined on 
$(-\infty,b(x))$ and is bounded from above by that
of $E/b(x)$ on this interval.

Let us now fix $\lambda<\overline{b}$. Thus there exists $z\geq 0$ such that 
$\lambda<b(x+z)$. The distribution of $ T_\PAR{x,+1}$ conditional on 
$T_\PAR{x,+1}>z$ is equal to the distribution of $z+ T_\PAR{x+z,+1}$ and 
the Laplace transform of $ T_\PAR{x,+1}$ can be split as follows
\begin{align*}
\dE\SBRA{e^{\lambda T_\PAR{x,+1}}}
&=\dE\SBRA{e^{\lambda T_\PAR{x,+1}}\ind_{T_\PAR{x,+1}\leq z}}
+\dE\SBRA{e^{\lambda T_\PAR{x,+1}}\ind_{T_\PAR{x,+1}>z}}\\
& 
=  \int_0^{z} b(x+t)e^{\int_{0}^{t}(\lambda-b(x+u))du}dt
+e^{\int_0^{z}(\lambda-b(x+u))du}\dE\SBRA{e^{\lambda T_\PAR{x+z,+1}}}.
\end{align*}
which is finite from the definition of $z$. 
\end{proof}


\begin{lem}[Stochastic order for jump times]\label{lem:EvolTx}
For $0\leq \tilde x<x$, we have
\[
T_{(x,+1)}\leqs T_{(\tilde x ,+1)} 
\quad\text{and}\quad 
T_{(x,-1)}\geqs T_{(\tilde x ,-1)}.
\]
\end{lem}

\begin{proof}
We first consider $v=+1$. 
Let $E$ be an exponential variable with unit mean and 
\[
T_\PAR{x,+1}:=B^{-1}\PAR{E+B(x)} -x
\quad\text{and}\quad
T_\PAR{\tilde x,-1}:=B^{-1}\PAR{E+B(\tilde x)} -\tilde x
\]
We have 
\[
\int_0^{ T_\PAR{x,+1}}b(x+s)ds=E= \int_0^{ T_\PAR{\tilde x,+1}}b(\tilde x+s)ds
\]
with $b$ a non-decreasing function and $x>\tilde x$, then clearly
$T_\PAR{x,+1}\leq T_\PAR{\tilde x,+1}$ (inequalities are strict when the 
jump rates are strictly monotone functions). The proof is similar for $T_{(x,-1)}$.
\end{proof}

We now give  helpful results in order  to construct a coalescent coupling of two processes 
starting from different initial data, which is in fact a generalization of a well known 
result on exponential random variables. Indeed, when $a<b$ are positive constants, 
 the following equalities  in distribution hold:
\begin{align*}
E(b)&\overset{\cL}{=}E(a)\wedge E(b-a),\\
E(a)&\overset{\cL}{=}E(b)+\varepsilon\, E'(a),
\end{align*}
where  random variables on the left hand side are independent and 
$\varepsilon$ is Bernoulli  with parameter $(1-a/b )$.

\begin{lem}[Decomposition of jump times, part I]\label{lem:comparTx+}
For $x\geq \tilde x \geq0$,  the identity in law
\[
T_\PAR{x,+1}\overset{\cL}{=}T_\PAR{\tilde x,+1}\wedge Z_+
\]
holds, with $Z_+$ a random variable with values  in $(0,\infty]$  independent  
of $T_\PAR{\tilde x,+1}$, such that
\begin{equation} \label{eq:survival}
\dP(Z_+>t)=\exp\PAR{-B(x+t)+B(x)+B(\tilde x+t)-B(\tilde x)}    \mbox{ for all }t\in[0,\tilde x). 
\end{equation}
Moreover,  there exists a coupling $(T_\PAR{x,+1}, T_\PAR{\tilde x,+1})$  such that, almost surely,
\[
T_\PAR{\tilde x,+1}\geq  T_\PAR{x,+1}
\]
and, conditionally on $\{T_{(x,+1)}=t\}$, 
\[
T_\PAR{\tilde x,+1}\overset{\cL}{=}  t  +\xi_{t}\hat{T}_\PAR{\tilde x+t,+1}
\]
with $\hat{T}_\PAR{\tilde x+t,+1} \overset{\cL}{=}   T_\PAR{\tilde x  +t,+1}$ 
and $\xi_t$  a Bernoulli r.v. independent of $\hat{T}_\PAR{x+t,+1}$ 
of parameter 
\[
\beta_t:=\frac{b(x+t)-b(\tilde{x}+t)}{b(x+t)}\in [0,1).
\]
\end{lem}
We observe that if $b(x)$ goes to $+\infty$ as $x\to\infty$, we have 
$Z_+<+\infty$ a.s.,  whereas $Z_+=+\infty$ a.s. if $b(x)=b$ is constant.

\begin{proof}
Since $b$ is a non-decreasing function, we have for any $t\geq 0$,
\[
B(t+x)-B(x)\geq B(t+\tilde x)-B(\tilde x).
\]
Using the representation \eqref{eq:T+} and the memoryless property of the exponential 
distribution, we thus have for all $t\geq 0$ that
\begin{align*}
\dP\PAR{T_\PAR{x,+1}>t}&=\dP\PAR{E>B(x+t)-B(x)}\\
&=\dP\PAR{T_\PAR{\tilde x,+1}>t}\dP\PAR{E>B(x+t)-B(x)-B(\tilde x+t)+B(\tilde x)}.
\end{align*}
The first statement follows. 
We next check that 
$(T_\PAR{x,+1}, T_\PAR{\tilde x,+1})
:=(T_\PAR{\tilde x,+1}\wedge Z_+, T_\PAR{\tilde x,+1})$,
with $(T_\PAR{\tilde x,+1}, Z_+)$ as before, is the required coupling. Since 
\begin{equation}\label{TZTT}
\{ T_\PAR{\tilde x,+1}>Z_+\}= \{T_\PAR{\tilde x,+1}> T_\PAR{ x,+1} \},
\end{equation}
we deduce that  
$T_\PAR{\tilde x,+1} =  T_\PAR{ x,+1}  
+  \left(T_\PAR{\tilde x,+1}-  T_\PAR{ x,+1}\right) 
\ind_{\{T_\PAR{\tilde x,+1}> T_\PAR{ x,+1} \}}.$ 
Thus, we just need to check that, conditionally on $\{T_{x,+1}=t\}$, 
\[
(\ind_{\{T_\PAR{\tilde x,+1}> T_\PAR{ x,+1} \}} ,   T_\PAR{\tilde x,+1}-  T_\PAR{ x,+1}) 
\overset{\cL}{=}  (\xi_t, \hat{T}_\PAR{\tilde x+t,+1} ). 
\]
Using \eqref{TZTT}, and the expression for the density of $T_\PAR{\tilde x,+1}$ 
together with \eqref{eq:survival}, we get 
  \begin{equation*}
\begin{split}
\dP \left( T_\PAR{\tilde x,+1}> T_\PAR{ x,+1} +r,  T_\PAR{ x,+1} >s \right) 
=& \, \dP \left( T_\PAR{\tilde x,+1}>  Z_+ + r, Z_+>s\right)	\\
= \, 
\int_s^{\infty}  \bigg[ \frac{b(x+t)-b(\tilde x +t)}{b(x+t)} 
&e^{- ( B(\tilde x+ r+t )- B(\tilde x +t ))}  \bigg]  e^{-(B(x+t)-B(x))}  b(x+t) dt \\
\end{split}
\end{equation*}
for all $s,r\geq 0$.  Alternatively, 
\begin{multline*}
\dP \left( T_\PAR{\tilde x,+1}> T_\PAR{ x,+1} +r,  T_\PAR{ x,+1} >s \right) = \\ 
\int_s^{\infty}  \dP \left(  T_\PAR{\tilde x,+1}- T_\PAR{ x,+1} >r,  
T_\PAR{\tilde x,+1}> T_\PAR{ x,+1} \vert   T_\PAR{ x,+1} =t\right) e^{-(B(x+t)-B(x))}  b(x+t) dt .
\end{multline*}
Taking derivative with respect to $s$ in the two above integrals, one concludes 
by  comparing the two different obtained expressions. 
\end{proof}

The function $a$ being non-increasing, we obtain an analogous result 
for $T_\PAR{x,-1}$:
\begin{lem}[Decomposition of jump times, part II]\label{lem:comparTx-} 
For $x> \tilde x >0$, the identity in law
\[
T_\PAR{\tilde x,-1}\overset{\cL}{=}T_\PAR{x,-1}\wedge Z_-
\]
holds, with $Z_-$ a random variable with values  in $(0,\tilde x]$  independent 
of $T_\PAR{x,+1}$, such that
\begin{equation}
\dP(Z_->t)=\exp\PAR{-A(\tilde x)+A(\tilde x-t)+A(x)-A(x-t)} \mbox{ for all }t\in[0,\tilde x). 
\end{equation}
Moreover,  there exists a coupling $(T_\PAR{x,-1}, T_\PAR{\tilde x,-1})$  such that, almost surely,
\[
T_\PAR{ x,-1}\geq  T_\PAR{\tilde x,-1}
\]
and, conditionally on $\{T_{(\tilde x,-1)}=t\}$, 
\[
T_\PAR{x,-1}\overset{\cL}{=}  t  +\chi_t\hat{T}_\PAR{x-t,-1}
\]
with $\hat{T}_\PAR{ x-t,-1} \overset{\cL}{=}   T_\PAR{ x  -t,+1} $ 
and   $\chi_t$  a Bernoulli r.v.  independent of $\hat{T}_\PAR{x-t,+1}   $  of parameter 
\[
\alpha_t:=\frac{a(\tilde x-t)-a(x-t)}{a(\tilde x-t)}\in [0,1) 
\quad \text{ if }t<\tilde x \quad \text{ and } \quad \alpha_{\tilde x}:=1.
\]
\end{lem}

 \begin{ex}[Explicit laws for jump times]\label{ex:loiT}\ 
In the case $b(x)=\underline b+x$ with $\underline b>0$ (as in the TCP model studied in \cite{robertetal,CMP}), 
 $T_\PAR{x,+1}$ has the density
\begin{equation*}
f_{\PAR{x,+1}}(t)=(\underline b+x+t)e^{{-\frac{(t+x+\underline b)^2-(x+ \underline b)^2}{2}}}\ind_\BRA{t>0}
\end{equation*}
and an everywhere finite Laplace transform given by  $\dE[e^{\lambda T_\PAR{x,+1}}]=1+\lambda \eta(x+\underline b-\lambda)$,
with $\eta(u)=e^{{u^2\over 2}}\sqrt{2\pi}\PAR{1-\Phi(u)}$ and $\Phi$ the 
cumulative distribution function of a standard Gaussian variable. We also notice in this case  that for
$0\leq\tilde x\leq x$,
\[
T_\PAR{x,+1}\overset{\cL}{=}T_\PAR{\tilde x,+1}\wedge E(x-\tilde x), 
\]
for $E(x-\tilde x)$  an exponential variable of mean $1/(x-\tilde x)$ 
independent of $T_\PAR{\tilde x,+1}$, and 
\[
\dP\PAR{E(x-\tilde x)>T_\PAR{\tilde x,+1}}=1-(x-\tilde x)\eta(x+ \underline b).
\]
\end{ex}

\subsection{Hitting time of the origin}\label{subsec:zero}

Let $(x,v)\in\dR_+\times\BRA{-1,+1}$. We notice that
\[
Z_\PAR{x,+1}\overset{\cL}{=}Z_\PAR{x,-1}+S_{x},
\]
where $Z_\PAR{x,v}$ is the first 
hitting time of $(0,+1)$ of a path starting from $(x,v)$  and $S_x$ is an excursion above $x$ independent of $Z_\PAR{x,-1}$. Consequently 
$Z_\PAR{x,-1}$ is stochastically smaller than $Z_\PAR{x,+1}$.

The  Laplace transform of the
hitting time of zero starting from $(x,v)$ was  explicitly computed  in \cite{FGM1} in the  case where $a$ and $b$ are both constant. Let us 
recall this result in the following proposition.

\begin{prop}[Hitting time of 0 for constant jump rates \cite{FGM1}]
Let us define $\lambda_c=\frac{1}{2}(\sqrt{b}-\sqrt{a})^2$ and, for any $\lambda\leq \lambda_c$,
\[
c(\lambda)=\frac{b-a-\sqrt{(a+b-2\lambda)^2-4ab}}{2} 
\quad\text{and}\quad
\psi(\lambda)=\frac{a+b-2\lambda-\sqrt{(a+b-2\lambda)^2-4ab}}{2a}.
\]
Then, for any $\lambda \in (-\infty,\lambda_c]$,
\[
 \dE\PAR{e^{\lambda Z_{(x,-1)}}}=
e^{xc(\lambda)}
\quad\text{and}\quad
 \dE\PAR{e^{\lambda Z_{(x,+1)}}}=
\psi(\lambda)e^{xc(\lambda)}.
\]
Moreover, these Laplace transforms are infinite on $(\lambda_c,\infty)$.
\end{prop}

If the jump rates $a$ and $b$ are not constant, the evolution away from the origin is no longer invariant by translation. Consequently, we have to consider a new way 
to estimate the distribution of the hitting time of zero.

\begin{prop}[Hitting time of 0 for general jump rates]\label{prop:zero}
Let $M>0$ such that 
\[
\sqrt{\frac{b(M)}{a(M)}}e^{ M\PAR{\sqrt{b(M)}-\sqrt{a(M)}}^2}\PAR{1-e^{-A(M)}}<1.
\]
Then, the  Laplace transform of the first hitting time $Z_\PAR{x,v}$ of $(0,+1)$ 
starting from $(x,v)\in \dR\times \{-1,+1\}$ satisfies
\[
\dE\SBRA{e^{\lambda Z_\PAR{x,v}}}\leq Ce^{\frac{(x\vee M)(b(M)-a(M))}{2}}  \, \mbox{    for  all } \lambda\leq \frac{1}{2}(\sqrt{b(M)}-\sqrt{a(M)})^2 ,
 \]
where  $C>0$ is an explicit constant depending on $M$,   $a$ and $b$.
\end{prop}

\begin{proof} We first notice that $f(x,v)=e^{\alpha x+\beta v}$ with $\alpha,\beta>0$, 
is a Lyapunov function for the  infinitesimal generator $A$ of ${((X_t,V_t))}_{t\geq 0}$ 
defined by \eqref{eq:giabs}.
More precisely,  we have \begin{align*}
&Af(x,+1)=f(x,+1)\SBRA{\alpha -b\PAR{x}\PAR{1-e^{-2\beta}}},\\
&Af(x,-1)=f(x,-1)\SBRA{-\alpha +a(x)\PAR{e^{2\beta}-1}}
\end{align*}
for all $x>0$. If we choose $\alpha,\beta>0$ and a compact set $K=[0,M]\times\{-1,1\}$ such that 
\begin{equation}\label{eq:betaM}
-\alpha+a(M)\PAR{e^{2\beta}-1}<0
\quad\text{and}\quad 
\alpha-b(M)\PAR{1-e^{-2\beta}}<0,
\end{equation}
by  monotony of $a$ and $b$ there are
$\rho=\rho(\alpha,\beta,M,a)>0$ to be specified and $c>0$ such that
\[
Af(x,v)\leq -\rho f(x,v)+c\ind_K(x,v).
\]
 The Laplace 
transform of the first hitting time  of $[0,M)\times\{-1,1\}$ starting from $(x,v)$, denoted 
$\tau_{(x,v)}:=\inf\{t> 0: (X_t,V_t)\in [0,M)\times\{-1,1\} \}
$, 
 can then  classically  be controlled  as follows: 
\begin{align*}
f(X_{t\wedge \tau_{(x,v)}},V_{t\wedge \tau_{(x,v)}})e^{\rho t}
&=f(x,v)+\int_0^{t\wedge \tau_{(x,v)}}\!\left[Af(X_s,V_s)+\rho f(X_s,V_s)\right]e^{\rho t}ds
+N_{t\wedge \tau_{(x,v)}}\\
&\leq f(x,v)+N_{t\wedge \tau_{(x,v)}}
\end{align*}
where $(N_t)_{t\geq 0}$ is a martingale with respect to the filtration generated 
by $((X_t,V_t))_{t\geq 0}$.
Taking the expectation in the previous inequality we deduce that, for $x>M$,
\begin{equation}\label{eq:LaplaceTau}
\dE\SBRA{e^{\rho \tau_{(x,v)}}}\leq e^{{\alpha (x-M)+\beta(v+1)}}.
\end{equation}
We next choose  $\alpha, \beta>0$  in order to optimize $\rho$. For $\alpha,\beta,M>0$ 
satisfying~\eqref{eq:betaM}, we set
\begin{align*}
\rho= \SBRA{b(M)(1-e^{-2\beta})-\alpha}\wedge\SBRA{\alpha-a(M)(e^{2\beta}-1)}.
\end{align*}
First we choose $\alpha>0$ such that
\[
2\alpha=b(M)(1-e^{-2\beta})+a(M)(e^{2\beta}-1)
\]
and then
\begin{equation}\label{eq:beta}
\beta=\frac{1}{4}\SBRA{\log\PAR{b(M)}-\log\PAR{ a(M)}}.
\end{equation}
With this choice, we have 
\begin{equation}\label{eq:rho}
2\alpha =b(M)-a(M)
\quad\text{and}\quad
\rho(M)=\frac{1}{2}\PAR{\sqrt{b(M)}-\sqrt{a(M)}}^2.
\end{equation}
Condition~\eqref{eq:betaM} is satisfied for any $M>0$ since 
\[
-\alpha+a(M)\PAR{e^{2\beta}-1}
=\alpha-b(M)\PAR{ 1-e^{-2\beta}}=-\frac{1}{2}\PAR{\sqrt{b(M)}-\sqrt{a(M)}}^2.
\]
We  now obtain an estimate for the Laplace transform of $Z_\PAR{x,+1}$. Let $M>0$ 
be arbitrarily fixed for the moment and $\lambda>0$ such that $\lambda\leq \rho(M)$. 
For $x\geq M$, we have 
\begin{equation}\label{eq:Zx}
Z_{(x,+1)}\overset{\cL}{=}\tau_{(x,+1)}+Z_{(M,-1)}
\end{equation}
where $\tau_{(x,+1)}$ and $Z_{(M,-1)}$ are independent.

\noindent Let us denote by $T_1$ and $T_2$ the first two jumps of $V$. From Figure~\ref{fi:cases} we see that, before a path  starting from $(M,-1)$ hits $(0,+1)$,   either    
\begin{enumerate}
 \item $T_1=M$, in which case we have 
$
Z_{(M,-1)}=M; 
$ 
 \item $T_1<M$ and $T_2>T_1$, in which case we have  
 \[
Z_{(M,-1)}=2T_1+Z_{(M,+1)}\leqs 2M+\tau_{(M,+1)}+\tilde Z_{(M,-1)}; \mbox{  or } 
\]  
 \item $T_1<M$ and $T_2\leq T_1$, and then
\[
Z_{(M,-1)}=T_1+T_2+Z_{(M-T_1+T_2,-1)}\leqs 2M+\tilde Z_{(M,-1)},
\]  
with $\tilde Z_{(M,-1)}$ an independent hitting time of zero starting from $(M,-1)$.
\end{enumerate}

\begin{figure}\label{fi:cases}
\begin{center}
\begin{tikzpicture}[xscale=1,yscale=1]
\draw[->] (0,0) -- (9.5,0);
\draw [->] (0,-0.5) -- (0,4);
\draw[dashed] (0,3) -- (5.9,3);
\draw (0,3) node[left] {\scriptsize{$M$}};
\draw [color=blue] (0,3) -- (2,1);
\draw [dashed,color=blue] (2,1) -- (3,0);
\draw [color=blue] (2,1) -- (3.2,2.2);
\draw [dashed,color=blue] (3.2,2.2) -- (4,3);
\draw [dashed,color=green] (4,3) -- (4.5,3.5);
\draw [dashed,color=green] (4.5,3.5) -- (5,3);
\draw [dashed,color=red] (5,3) -- (6.5,1.5);
\draw [dashed,color=red]  (6.5,1.5) -- (7,2);
\draw [dashed,color=red]  (7,2) -- (9,0);
\draw [dashed] (3.2,2.2) -- (5.9,2.2);
\draw [dashed,color=red] (3.2,2.2) -- (3.9,1.5);
\draw [dashed,color=red] (3.9,1.5) -- (4.6,2);
\draw [dashed,color=red] (4.6,2) -- (6.6,0);
\draw (2.8,0.5) node {\scriptsize{$1.$}};
\draw (3.3,2.7) node {\scriptsize{$2.$}};
\draw (3.8,1.9) node{\scriptsize{$3.$}};
\draw [dashed] (2,1) -- (2,0);
\draw (2,0) node[below] {\scriptsize{$T_1$}};
\draw [dashed] (3.2,0) -- (3.2,2.2);
\draw (3.4,0) node[below] {\scriptsize{$T_2$}};
\end{tikzpicture}
\end{center}
\caption{The three different types of paths from $(M,-1)$ to $(0,+1)$.}
\end{figure} 

\noindent Lemma~\ref{lem:law-time} ensures that $\dP_{(M,-1)}(T_1=M)=e^{-A(M)}$. As a consequence, 
if $\varphi$ is the Laplace transform of $Z_{(M,-1)}$, one has 
\[
\varphi(\lambda)\leq e^{-A(M)} e^{\lambda M}
+\PAR{1-e^{-A(M)}} e^{2\lambda M} \dE\SBRA{e^{\lambda \tau_{(M,+1)}}}
\varphi(\lambda).
\]
For any $\lambda<\rho(M)$ we get, thanks to Hölder inequality and \eqref{eq:LaplaceTau}, that
\[
\dE\SBRA{e^{\lambda \tau_{(M,+1)}}}
\leq \PAR{\dE\SBRA{e^{\rho(M) \tau_{(M,+1)}}}}^{\lambda/\rho(M)} 
\leq e^{2\lambda\beta(M)/\rho(M)}.
\]
Thanks to~\eqref{eq:beta}, if $M$ is chosen in order  that 
\[
\sqrt{\frac{b(M)}{a(M)}}e^{ M\PAR{\sqrt{b(M)}-\sqrt{a(M)}}^2}\PAR{1-e^{-A(M)}}<1
\]
then for any $\lambda\leq \rho(M)$, $\varphi$ is finite and
\[
\varphi(\lambda)\leq 
\frac{e^{-A(M)} e^{\lambda M}}{1-\PAR{1-e^{-A(M)}}e^{2\lambda(M+\beta(M)/\rho(M))}}.
\]
Combining \eqref{eq:Zx} and \eqref{eq:LaplaceTau} with the previous estimate completes the proof 
in the main case $x\geq M$.  
If $x\leq M$ then 
\[
Z_{(x,-1)}\leqs Z_{(M,-1)}\quad\text{ and }\quad Z_{(x,+1)}\leqs M+\tau_{(M,+1)}+\tilde Z_{(M,-1)},
\]
and one can conclude as in the previous case.  
\end{proof}

\section{The coupling time for the reflected process}\label{sec:coalescent-refl}

This section is dedicated to the construction of a coalescent coupling of 
two paths of the reflected process driven by~\eqref{eq:giabs} starting from two different initial 
conditions. 

\subsection{The first crossing time}\label{subsec:crossing}

Let us consider $(x,v)$ and $(\tilde x,\tilde v)$ two initial data with $\tilde x<x$. The 
first crossing time of two paths $(X,V)$ and $(\tilde X,\tilde V)$ starting respectively 
from $(x,v)$ and $(\tilde x,\tilde v)$ is defined by 
\[
T_c=T_c(x,v,\tilde x,\tilde v)=\inf\BRA{t\geq 0\ :\ X_t=\tilde X_t}. 
\]
Since $(X_t)_{t\geq 0}$ is continuous, $T_c$ is stochastically smaller than  
  the hitting time of zero  $ Z_\PAR{x,v}$ of the initially upper path,  whatever the joint law of the pair. The first 
 crossing point $X_{T_c}$ is such that 
\begin{equation}\label{eq:CrossPoint}
X_{T_c}\, \leq \sup_{t\in[0,Z_\PAR{x,v}]}X_t -(x-\tilde x)
\, \leq \, {1\over 2}\PAR{Z_\PAR{x,v}+\tilde x-x}.
\end{equation}
 Notice that at time $T_c$ the  two velocities are opposites. 

\subsection{A way to stick the two paths}\label{sec:stick}

In what follows we assume that $(X_0,V_0)=(x,+1)$ and
$(\tilde X_0,\tilde V_0)=(x,-1)$ and construct  two paths which
are equal after a coalescent time $T_{*}(x)$. The successful coupling 
consists in producing a jump of (exactly) one of the two velocities $V$ or $\tilde V$ 
at a crossing time of their position components $X$ and $\tilde X$. We will use  the  coupled jump-times studied in Lemmas \ref{lem:comparTx+} and \ref{lem:comparTx-} in order to minimize the time required to do so, in the spirit of  \cite{FGM1}.

To be more precise, given  $x>0$ fixed, let us   denote by $U_+$ 
and $U_-$ (respectively $L_-$ and $L_+$) the first and the second inter-jump time lapses
of the  path starting from $(x,+1)$ (resp. starting from $(x,-1)$). These random variables 
are constructed 
as follows. We first choose $U_+$  with distribution $T_\PAR{x,+1}$ and $L_-$ 
with distribution $T_\PAR{x,-1}$ independently. We then define $U_-$ and $L_+$ 
in such a way that $(U_+,L_+)$ and $(U_-,L_-)$ have  the laws of the  couplings 
defined in Lemmas \ref{lem:comparTx+} and \ref{lem:comparTx-} respectively 
and that $L_+ -U_+$ and $U_-  - L_- $ are independent conditionally on 
$(U_+,L_-)$. More precisely, conditionally on  $U_+$ and $L_-$, we introduce 
two independent Bernoulli variables $\xi$ and $\chi$  with 
\begin{equation*}
\begin{split}
\dP(\xi=1\vert U_+,L_-)= & \, \frac{b(x+U_+)-b(x-L_-+U_+)}{b(x+U_+)}  
\quad \mbox{ and }\\
\dP(\chi=1\vert U_+,L_-) = &  
\, \frac{a(x-L_-)-a(x+U_+-L_-)}{a(x-L_-)}\ind_\BRA{L_-<x}+ \ind_\BRA{L_-=x}\\
\end{split}
\end{equation*}
and two independent random variables  $L_+ -U_+$ and  $U_-  - L_-$  with 
the same law as $T_{x+U_+-L_-,+1}$  and $T_{x+U_+-L_-,-1}$ respectively. 
Then we set
\[
L_+:=U_+ + \xi(L_+ -U_+)  \quad  \text{    and  }  \quad U_-:= L_- +\chi  (U_-  - L_- ).
\]
Figure~\ref{fi:coupling} shows the four possible outcomes. Those where exactly one 
of the Bernoulli variables is equal to $1$ allow us to  stick the paths at time $U_+ + L_-$ 
(i.e. on  the rightmost corner of the rectangle):  the velocities of the two paths  are the 
same right after that instant, and the  overshot length (beyond the  rectangle's corner) 
determined by the previous  coupling  is compatible with the law of the two marginal 
processes from that moment on (because of their Markov property). We then say 
that the coupling attempt succeeded, and this happens conditionally on $(U_+,L_-)$ 
with probability
\begin{equation}\label{eq:defp} 
\begin{split} \dP & (\xi=0,\chi=1\vert U_+,L_-)+\dP(\xi=1,\chi=0\vert U_+,L_-)=\\
\, \,  & \left[ \frac{b(x-L_-+U_+)}{b(x+U_+)}\PAR{1-\frac{a(x-L_-+U_+)}{a(x-L_-)}}
+ \PAR{1-\frac{b(x-L_-+U_+)}{b(x+U_+)}}\frac{a(x-L_-+U_+)}{a(x-L_-)} \right] \ind_\BRA{L_-<x}\\
 &  +  \frac{b(U_+)}{b(x+U_+)}\ind_\BRA{L_-=x}. \\ 
 \end{split}
\end{equation}

Observe that the success or failure of the coupling attempt  is determined by 
the Bernoulli random variables $\xi$ and $\chi$. If the coupling attempt fails, 
the two trajectories cross or bounce off of each other at  time $U_+ + L_-$ 
and by similar reason as before the (already determined) lengths $(L_+ -U_+)$ 
and  $(U_-  - L_- )$ can be used to restart two (upward and downward) trajectories 
from  $x-L_-+U_+$, independently of each other conditionally on the past and 
consistently with the pathwise laws of each of the two processes. 

The coupling construction is now obvious:  we repeat this scheme starting 
from the new crossing or bouncing point, until we succeed in sticking the two paths.
Notice that this iterative algorithm is more efficient than the general procedure 
of the Meyn-Tweedie method~\cite{MT2} since, after a fail, the two processes have 
already the same position (and still opposite velocities). 

\begin{figure}
\begin{center}
\begin{tikzpicture}[xscale=0.45,yscale=0.45]
\begin{scope}
\draw[->] (0,0) -- (7,0);
\draw [->] (0,-0.5) -- (0,7);
\draw (0,3.2) node[left] {\scriptsize{$x$}};
\draw [color=green](6.8,5) node[left] {\scriptsize{$X_t$}};
\draw [color=blue](6.8,2) node[left] {\scriptsize{$\tilde X_t$}};
\draw [color=green] (0,4) -- (2,6);
\draw [color=blue] (0,4) -- (3,1);
\draw [color=green] (2,6) -- (5,3);
\draw [color=blue] (3,1) -- (5,3);
\draw [dashed] (5,0) -- (5,3);
\draw (5,0) node[below ] {\scriptsize{$T_{cc}$}};
\draw [color=green] (5,3) -- (5.5,3.5);
\draw [color=blue] (5,3) -- (5.5,3.5);
\draw (0.8,5) node[above ] {\scriptsize{$U_+$}};
\draw (4,4.5) node[above ] {\scriptsize{$U_-$}};
\draw (1.3,2.6) node[below ] {\scriptsize{$L_-$}};
\draw (4.2,1.7) node[below ] {\scriptsize{$L_+$}};
\draw (5,-1) node[below ] {\scriptsize{Case 1: $U_+<L_+$ and $U_-=L_-$}};
\end{scope}

\begin{scope}[xshift=12.5cm,yshift=0cm]
\draw[->] (0,0) -- (7,0);
\draw [->] (0,-0.5) -- (0,7);
\draw (0,3.2) node[left] {\scriptsize{$x$}};
\draw [color=green](6.8,5) node[left] {\scriptsize{$X_t$}};
\draw [color=blue](6.8,2) node[left] {\scriptsize{$\tilde X_t$}};
\draw [color=green] (0,4) -- (2,6);
\draw [color=blue] (0,4) -- (3,1);
\draw [color=green] (2,6) -- (5,3);
\draw [color=blue] (3,1) -- (5,3);
\draw [dashed] (5,0) -- (5,3);
\draw (5,0) node[below ] {\scriptsize{$T_{cc}$}};
\draw [color=green] (5,3) -- (5.5,2.5);
\draw [color=blue] (5,3) -- (5.5,2.5);
\draw (5,-1) node[below ] {\scriptsize{Case 2: $U_+=L_+$ and $U_->L_-$}};
\end{scope}

\begin{scope}[xshift=0cm,yshift=-9.5cm]
\draw[->] (0,0) -- (7,0);
\draw [->] (0,-0.5) -- (0,7);
\draw (0,3.2) node[left] {\scriptsize{$x$}};
\draw [color=green](6.8,5) node[left] {\scriptsize{$X_t$}};
\draw [color=blue](6.8,2) node[left] {\scriptsize{$\tilde X_t$}};
\draw [color=green] (0,4) -- (2,6);
\draw [color=blue] (0,4) -- (3,1);
\draw [color=green] (2,6) -- (5,3);
\draw [color=blue] (3,1) -- (5,3);
\draw [dashed] (5,0) -- (5,3);
\draw (5,0) node[below ] {\scriptsize{$T_c$}};
\draw [color=green] (5,3) -- (5.5,2.5);
\draw [color=blue] (5,3) -- (5.5,3.5);
\draw (5,-1) node[below ] {\scriptsize{Case 3: $U_+<L_+$ and $U_->L_-$}};
\end{scope}

\begin{scope}[xshift=12.5cm,yshift=-9.5cm]
\draw[->] (0,0) -- (7,0);
\draw [->] (0,-0.5) -- (0,7);
\draw (0,3.2) node[left] {\scriptsize{$x$}};
\draw [color=green](6.8,5) node[left] {\scriptsize{$X_t$}};
\draw [color=blue](6.8,2) node[left] {\scriptsize{$\tilde X_t$}};
\draw [color=green] (0,4) -- (2,6);
\draw [color=blue] (0,4) -- (3,1);
\draw [color=green] (2,6) -- (5,3);
\draw [color=blue] (3,1) -- (5,3);
\draw [dashed] (5,0) -- (5,3);
\draw (5,0) node[below ] {\scriptsize{$T_c$}};
\draw [color=green] (5,3) -- (5.5,3.5);
\draw [color=blue] (5,3) -- (5.5,2.5);
\draw (5,-1) node[below ] {\scriptsize{Case 4: $U_+=L_+$ and $U_-=L_-$}};
\end{scope}
\end{tikzpicture}
\end{center}
\caption{Position of both paths after one step.}
\label{fi:coupling}
\end{figure}
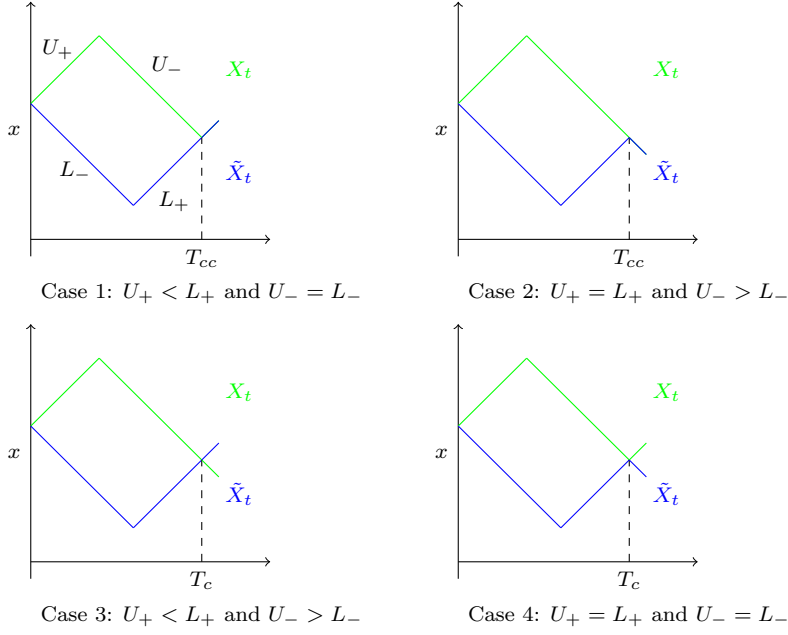

We point out that the coupling scheme implemented for the reflected process 
in \cite{FGM1} in the constant jump rates case is a particular case of the above 
described scheme. Notice however that here, in general, the upper  path starting 
from $(x,+1)$ does not necessarily remain 
above the other path until the coupling time. We then cannot control the coupling 
time by the hitting time of $0$ for the process starting at $(x,+1)$ as it was done 
in \cite{FGM1}. On the other hand, contrary to the constant rates case where the 
coupling  could only succeed right after  the lower process hit $0$, the coupling 
can now succeed at an arbitrary step of the scheme, though not with a probability 
bounded from below uniformly in $x$ (this can be easily seen from 
formula~\eqref{eq:defp} e.g. in the case case when $a$ is constant and 
$b(x)=a+x$). Therefore, a new approach to estimate the coupling time is needed, 
which is developed in next subsection.

\subsection{Coupling time from a crossing point}
\label{sec:CouplingTime}

In this section we will use the notation $\dP_x$ (resp. $\dE_x$) for the distribution 
(resp. the expectation) of a random variable associated with the coupling scheme 
given that  the two copies started at position $x>0$.

\noindent We first observe that  for  fixed $R>0$,  the probability of success in one step 
can be bounded from below  (considering the last term in \eqref{eq:defp} and 
taking expectation) uniformly over $x\in[0,R]$ by some number $p_R\in (0,1)$ 
satisfying
\begin{equation}\label{eq:pR}
p_R\geq  e^{-A(R)}\int_0^\infty b(u)e^{-\int_0^{u}b(R+s)ds}du.
\end{equation} 
This suggests that the number of trials below a fixed height $R>0$ required in 
order to get a successful coupling can be stochastically  dominated by some geometric random 
variable. Notice that we do not expect a successful coupling to occur only below 
level $R$. We will rather use the above remark in order to construct the scheme 
in such a way that the coalescent time will be always smaller than or equal to some 
real random variable that we can control in terms of  geometric number of 
positive time lapses. 

First, we define a sequence of ``rectangles'' of potential  trajectories of the two 
copies in the coupling scheme, on which  the two copies will live at all times, 
irrespective of whether the coupling attempt has  already been successful or 
not (of course once it has been so, their positions and velocities coincide from that moment 
on).  More precisely, we define a discrete time Markov chain $(\Phi_n)_{n\geq 0}$ 
starting at $x$ by
\begin{equation}\label{eq:phi}
\begin{split}
\Phi_0=& \, x, \\
\Phi_{n+1} = & \, \Phi_n + T^{n+1}_{( \Phi_n,+1)}- T^{n+1}_{( \Phi_n,-1)},\\
\end{split}
\end{equation}
where conditionally on all the past  up to (and including) time $n$, 
$T^{n+1}_{( \Phi_n,+1)}$ and $T^{n+1}_{( \Phi_n,-1)}$ are independent 
and respectively equal in law to $T_{(y,+1)}$  and $T_{(y,-1)}$ on the event 
$\BRA{\Phi_{n}=y}$. Plainly, $(\Phi_n)_{n\geq 1}$ describes  the height of the 
right-most corner of the $n-$th rectangle obtained by iterating the construction 
of Figure~\ref{fi:coupling}. Consider also the sequence of positive random variables 
(real time lengths) $(\sigma_n)_{n\geq 0}$ defined by 
\begin{equation}\label{eq:sigma}
\begin{split}
\sigma_0=& 0, \\
\sigma_{n+1} = &  T^{n+1}_{( \Phi_n,+1)}+  T^{n+1}_{( \Phi_n,-1)},
\end{split}
\end{equation}
which give the (real) time-position of the rectangles'  right-most corners, and 
finally set  $\Sigma_{n}=\sum_{i=1}^{n} \sigma_i $, with the convention $\Sigma_{0}=0$. 
Following Lemmas \ref{lem:comparTx+} and \ref{lem:comparTx-} and in order 
to determine at which attempt the coupling is successful, we introduce two 
sequences $(\xi_n)_{n\geq 1}$ and $(\chi_n)_{n\geq 1}$ of Bernoulli random 
variables, conditionally independent of each other given 
$\PAR{\sigma_k,\Phi_k}_{k\geq 0}$  and  such that for $n\geq 0$, 
on $\BRA{\Phi_{n}=y,T^{n+1}_{( \Phi_n,+1)}=t,T^{n+1}_{( \Phi_n,-1)}=s}$, 
\begin{equation*}
\begin{split}
\dP(\xi_{n+1}= & \, 1\vert \PAR{\sigma_k,\Phi_k}_{k\leq n +1} )
=\frac{b(y+t)-b(y-s+t)}{b(y+t)},\quad \mbox{ and }\\
\dP(\chi_{n+1}=& \, 1\vert  \PAR{\sigma_k,\Phi_k}_{k\leq n+1}  )
=\frac{a(y-s)-a(y+t-s)}{a(y-s)}\ind_\BRA{s<y}+ \ind_\BRA{s=y}.
\end{split}
\end{equation*}
We also set  $\xi_0=\chi_0=1$ for notational simplicity. Observe that $\PAR{\sigma_n,\Phi_n, \xi_n,\chi_n}_{n\geq 0}$ is  Markov process. We denote by 
$\PAR{\cF_n}_{n\geq 0}$ the filtration  it generates. 
We then define a sequence of  random variables $(\kappa_n)$ by 
\[
\kappa_n=\ind_\BRA{\xi_n=1,\chi_n=0}+\ind_\BRA{\xi_n=0,\chi_n=1}.
\]
According to the discussion at the end of the previous subsection, 
the discrete time instant  (or rectangle number) at which the coupling succeeds is
\[
\rho:=\inf\BRA{n\geq 1: \kappa_n=1}
\]
and the real time spent in order that this happens is
\[
T_*:=\Sigma_\rho=\sum_{i=1}^{\rho}\sigma_i.
\]
The trajectories of the two copies can be easily constructed from the previous objects, but we actually  do not need to work with them.

\noindent Let us now  introduce the discrete random variable
\[
\rho_R:=\inf\BRA{n\geq 1: \kappa_n=1 \text{ and } \Phi_{n-1}<R}.
\]
Both $\rho$ and $\rho_R$ are stopping times with respect to $\PAR{\cF_n}_{n\geq 0}$. 
Since $\rho\leq \rho_R$ a.s., we clearly have
\begin{equation}\label{eq:majcoupling}
T_*\leq T_*^R :=\sum_{i=1}^{\rho_R}\sigma_i.
\end{equation}

%
%
%
%

Our goal now is to exhibit an  upper bound  for the Laplace transform of the 
random time $T_*^R$ under $\dP_x$. We need to introduce  the stopping 
time (with respect to  $\PAR{\cF_n}_{n\geq 0}$)
\begin{equation}\label{tauRx}
\tau_R(x):=\inf\{ n\geq  0 \, : \, \Phi_n \in [0,R]\} ,
\end{equation}
 and the real time 
\begin{equation}\label{SigmatauRx}
 \Sigma_{\tau_R}(x):=\sum_{i=0}^{\tau_R(x)} \sigma_i 
 \end{equation}
 accumulated when the sequence $(\Phi_n)_{n\geq 0}$ reaches $[0,R]$ for the 
 first time.  Let $\varphi_\PAR{x,v}$ denote  the Laplace transform of 
 $T_\PAR{x,v}$ for $(x,v)\in\dR^+\times\BRA{-1,+1}$. We have

\begin{lem}\label{lem:taR}
Assume there exist  positive real numbers  $R, \lambda, \beta$  such that 
$\lambda<\beta$, $\lambda+\beta<\bar b$ and  
$\varphi_\PAR{R,+1}(\beta+\lambda) \varphi_\PAR{R,-1}(\lambda-\beta)  <1$. 
Then, $ \tau_R<\infty$  a.s. and 
  \begin{equation*}
 \dE_x\left[e^{\beta \Phi_{\tau_R}+\lambda  \sum_{i=0}^{\tau_R} \sigma_i}\eta^{\tau_R}\right]
 \leq e^{\beta x  \ind_{x> R}}
 \end{equation*}
 where $\eta:= (\varphi_\PAR{R,+1}(\beta+\lambda) \varphi_\PAR{R,-1}(\lambda-\beta))^{-1}>1$. 
  \end{lem}

\begin{proof}



For each $x>R$ and $\beta>\lambda$, from  the stochastic  monotonicity of the jump times 
(see Lemma~\ref{lem:EvolTx}) we get
\begin{equation}\label{Lyaputypeestim}
\begin{split}
\dE_x\left[ e^{\beta \Phi_1 +\lambda \sigma_1}\right] = &  
e^{\beta x} \dE\left[ e^{(\beta  +\lambda) T_{x,+1}} \right]  
\dE\left[ e^{(\lambda- \beta  ) T_{x,-1}} \right] \\
 \leq &   e^{\beta x}\varphi_\PAR{R,+1}(\beta+\lambda) \varphi_\PAR{R,-1}(\lambda-\beta) .
 \end{split}
 \end{equation}
 If $\varphi_\PAR{R,+1}(\beta+\lambda) \varphi_\PAR{R,-1}(\lambda-\beta)  <1$, 
 we deduce from \eqref{Lyaputypeestim} that 
 $e^{\beta \Phi_{\tau_R\wedge n}+   \lambda  \sum_{i=1}^n \sigma_i}\eta^{\tau_R\wedge n}$ 
 is a positive supermartingale with respect to $\PAR{\cF_n}_{n\geq 0}$, hence
  \begin{equation*}
 e^{\beta x  \ind_{x> R}}\geq  
 \dE_x\left[  e^{\beta \Phi_{\tau_R\wedge n } +   
\lambda  \sum_{i=0}^{\tau_R\wedge n }  \sigma_i} \eta^{\tau_R\wedge n }\right]
\geq \dE_x\left[  \eta^{\tau_R\wedge n }\right].
 \end{equation*}
 Letting $n\to \infty$ in the last expectation we get by monotone convergence 
 $\dE_x\left[  \eta^{\tau_R }\right]<\infty$, hence $\tau_R<\infty $ a.s..  
 Letting then $n\to \infty$ in the first expectation and using Fatou's Lemma the statement  follows.
 \end{proof}
 
For each $\gamma>0$ and $R>0$, we  now set
\[
{\cal E}_R(\gamma):= 
\sup_{y\in [0,R]} \dE_y\SBRA{e^{\gamma T_\PAR{y,+1}}\ind_{\kappa_{1}=0}}.
\]

\begin{prop}[Laplace transform of the coupling time starting at a crossing point]\label{prop:stick}
Assume that $(R,\lambda,\beta)$ satisfy the conditions of Lemma \ref{lem:taR} and 
moreover  that  $ {\cal E} _R(\lambda +\beta)<1$. Then,  the Laplace transform of $T_*$ 
satisfies
 \begin{equation}\label{eq:CouplingTime}
 \dE_x\SBRA{e^{\lambda T_*}}\leq 
 \frac{e^{\beta x \ind_{x> R}}
 \varphi_\PAR{0,+1}(\lambda)\varphi_\PAR{R,-1}(\lambda)}{1-  {\cal E} _R(\lambda +\beta) }.
 \end{equation}
 \end{prop}

\begin{proof} 
Fix $x\in\dR$ and $R, \lambda, \beta$ satisfying conditions of Lemma~\ref{lem:taR}. 
From \eqref{eq:majcoupling}, we just need to estimate $T_*^R$. To this aim, 
we consider the process $(\hat\Phi_n, \hat\sigma_n)_{n\geq 0}$ defined in terms of 
$(\Phi_n, \sigma_n)_{n\geq 0}$ in the following way: first set $\hat\tau_0=0$ and  
$\hat\tau_1=\tau_R(x)+1$ and for all $n\geq 1$, 
\[
\hat\tau_{n+1}=\hat\tau_n+\tau^n_R+1,
\quad  \text{with}\quad 
\tau_R^n:=\inf\{ n\geq  \hat\tau_n \, : \, \Phi_n \in [0,R]\}-  \hat\tau_n. 
\]
In other words, $\hat\tau_{n+1}$ is the index of the first attempt to couple the paths that 
follows the first (discrete) return time $\hat\tau_n+\tau^n_R$  of $(\Phi_k)_{k\geq 0}$ 
into $[0,R]$ after  $\hat\tau_n$. Then, we set $\hat \Phi_0=x$, $\hat\sigma_0=0$, 
$\hat\Phi_1=\Phi_{\hat\tau_1}$, $\hat\sigma_1=\sum_{i=1}^{\hat\tau_1}\sigma_{i}$ and
\[
\hat\Phi_{n+1}=\Phi_{\hat\tau_{n+1}}, \quad\hat\sigma_{n+1}
=\sum_{i=\hat\tau_n+1}^{\hat\tau_{n+1}}\sigma_{i}.
\]
Thus, $\hat\sigma_{n+1}$ is the sum of the real time 
\[
\Sigma^n_{\tau^n_R}:=\sum\limits_{i=\hat\tau_n+1}^{\hat\tau_{n+1}-1}\sigma_i
\]
needed after $\hat\sigma_{n}$ in order to observe again a ``rectangle corner'' in 
$[0,R]$, plus the time   $\sigma_{\hat\tau_{n+1}}$ spent in one coupling attempt 
right thereafter. Then,  $\hat\Phi_{n+1}\in \dR_+$ is the position of the discrete 
chain (or rectangle corner) at the  time instant $\hat\sigma_{n+1}$. Notice that 
for each  $i\geq 0$,  $\hat\tau_i$ is a stopping time with respect to the filtration 
$\PAR{\cF_n}_{n\geq 0}$ and that, conditionally on   
$\cF_{\hat\tau_n}\cap \{ \hat\Phi_n = x\}$, $(\tau^n_R,\Sigma^n_{\tau^n_R})$ 
has the same law as  the pair $(\tau_R(x) ,\Sigma_{\tau_R(x)})$  defined in 
\eqref{tauRx} and \eqref{SigmatauRx}. We can now write
\[
T_*^R=\sum_{i=0}^{\hat\rho_R}\hat\sigma_i,
\]
where $\hat\rho_R:=\inf\BRA{n\geq 1:\kappa_{\hat\tau_n}=1 }$ is a stopping 
time with respect to the filtration $\PAR{  \cF_{\hat\tau_n} }_{n\geq 0} $.  

We notice that $\hat\rho_R<\infty$ a.s. since the probability of fail in one step 
starting from a position $x\in[0,R]$ is uniformly bounded on $[0,R]$ by $1-p_R$.
We then can write
\begin{equation}\label{eq:LTR}
\dE_x[e^{\lambda T_*^R}]
=\sum_{ n=1}^\infty\dE_x\SBRA{e^{\lambda \sum_{i=0}^{n}\hat\sigma_i}\ind_{\hat\rho_R=n}}
=\sum_{ n=1}^\infty\dE_x\SBRA{e^{\lambda \sum_{i=0}^{n}\hat\sigma_i}\ind_{\kappa_{\hat\tau_n}
=0,\kappa_{\hat\tau_{n-1}}=0,\ldots,  \kappa_{\hat\tau_0}=1}}.
\end{equation}
On one hand, we have
 \begin{align*}
 \dE_x [e^{\lambda \hat\sigma_{1}}]= \, 
 & \dE \SBRA{e^{\lambda\Sigma_{\tau_R}(x)}e^{\lambda \sigma_{\tau_R(x)+1}}}
 = \dE \SBRA{\dE\SBRA{e^{\lambda  \sigma_{\tau_R(x)+1}}
 \vert {\cal F}_{\tau_R(x)} }e^{\lambda\Sigma_{\tau_R}(x)}} \\
\leq \,  &   \dE \SBRA{\dE_{\Phi_{\tau_R(x)}} 
\SBRA{e^{\lambda  \sigma_1} }e^{\lambda\Sigma_{\tau_R}(x)}}.
\end{align*}
By  Lemma \ref{lem:EvolTx}, 
$\dE_y\SBRA{e^{\lambda \sigma_1}}
=\varphi_\PAR{y,+1}(\lambda)\varphi_\PAR{y,-1}(\lambda)
\leq\varphi_\PAR{0,+1}(\lambda)\varphi_\PAR{R,-1}(\lambda)$ 
for any point $y\in[0,R]$. Thus, using also Lemma \ref{lem:taR} we get
\begin{equation*}
 \dE_x [e^{\lambda \hat\sigma_{1}}] \leq 
 \varphi_\PAR{0,+1}(\lambda)\varphi_\PAR{R,-1}(\lambda)e^{\beta x \ind_{x> R}},
\end{equation*}
and
\[
\dE\SBRA{e^{\lambda \hat\sigma_n}\ind_{\kappa_{\hat\tau_n}=1}\vert\mathcal{G}_{n-1}}
=\dE_{\hat\phi_{n-1}}\SBRA{e^{\lambda \hat\sigma_1}\ind_{\kappa_{\hat\tau_1}=1}}
\leq \varphi_\PAR{0,+1}(\lambda)\varphi_\PAR{R,-1}(\lambda)e^{\beta \hat\phi_{n-1} 
\ind_{\hat\phi_{n-1}> R}}.
\]
On the other hand,  we have
\begin{equation*}
\dE_x\SBRA{e^{\lambda \hat\sigma_1}e^{\beta \hat\phi_{1} \ind_{\hat\phi_{1}> R}}
\ind_{\kappa_{\hat\tau_1}=0}}
=\dE\SBRA{e^{\lambda \Sigma_{\tau_R}(x)}
\dE_{\Phi_{\tau_R}(x)}\SBRA{e^{\lambda \sigma_1
+\beta \phi_{1} \ind_{\phi_{1}> R}}\ind_{\kappa_{1}=0}}}
\end{equation*}
and for all $y\in[0,R]$
\begin{equation*}
\dE_{y}\SBRA{e^{\lambda \sigma_1+\beta \phi_{1} \ind_{\phi_{1}> R}}\ind_{\kappa_{1}=0}}
\leq  {\cal E} _R(\lambda +\beta) e^{\beta y}.
\end{equation*}
Then, again from Lemma \ref{lem:taR} we get
\begin{equation*}
\dE_x\SBRA{e^{\lambda \hat\sigma_1}e^{\beta \hat\phi_{1} \ind_{\hat\phi_{1}> R}}
\ind_{\kappa_{\hat\tau_1}=0}}\leq   {\cal E} _R(\lambda +\beta)  e^{\beta x \ind_{x> R}}
\end{equation*}
and then,  for all $k=1,\dots, n-1$,
\begin{align*}
\dE\SBRA{e^{\lambda \hat\sigma_k}e^{\beta \hat\Phi_{k} \ind_{\hat\phi_{k}> R}}
\ind_{\kappa_{\hat\tau_k}=0}\vert\mathcal{G}_{k-1}}&=\dE_{\hat\phi_{k-1}}
\SBRA{e^{\lambda \hat\sigma_1}e^{\beta \hat\phi_{1} \ind_{\hat\phi_{1}> R}}
\ind_{\kappa_{\hat\tau_1}=0}}\\
&\leq  {\cal E} _R(\lambda +\beta)  e^{\beta \hat\phi_{k-1} \ind_{\hat\phi_{k-1}> R}}.
\end{align*}
\noindent  By successively conditioning   in \eqref{eq:LTR}, we finally have
\begin{align*}
\dE_x[e^{\lambda T_*^R}]&\leq e^{\beta x \ind_{x> R}}
\sum_{n=1}^\infty \varphi_\PAR{0,+1}(\lambda)
\varphi_\PAR{R,-1}(\lambda)\PAR{  {\cal E} _R(\lambda +\beta) }^n\\
&= \frac{e^{\beta x \ind_{x> R}}
 \varphi_\PAR{0,+1}(\lambda)\varphi_\PAR{R,-1}(\lambda)}{1- {\cal E} _R(\lambda +\beta)}
\end{align*}
for parameters as required. 
\end{proof}

 \bigskip

Let us now verify  the existence of $\PAR{R,\lambda,\beta}$ such that all the 
assumptions of Proposition~\ref{prop:stick} hold. Notice first that for all
   $\beta>\lambda>0$ and each $R>0$   we have
 \begin{equation}\label{eq:phi-}
  \varphi_\PAR{R,-1}(\lambda-\beta) \leq  
  \frac{\bar{a} + (\beta -\lambda) e^{-(\bar{a} +\beta -\lambda)R} }{\bar{a} +\beta-\lambda },
  \end{equation}
  thanks to the fact that  
$T_\PAR{R,-1}\geq_{\mathrm{sto.}}E(\overline{a})\wedge R$. Since also 
$T_\PAR{R,+1}\leqs E(b(R))$,   we furthermore have
 \begin{equation}\label{eq:phi+}
  \varphi_\PAR{R,+1}(\lambda+\beta) \leq \frac{b(R)}{b(R)-(\lambda +\beta) }
  \end{equation}
for all  $\lambda+\beta<b(R)$.

Given $\lambda >0$, we take $\beta>\lambda$ of the form $\beta=\alpha \lambda$ for 
$\alpha>{\bar b+\bar a\over \bar b -\bar a}$  (or simply $\alpha>1$ if $\bar b=\infty$). 
Then, we have  $\bar b>{\alpha+1\over \alpha -1}\bar a$, hence we find $R$ large 
enough such that
\[
b(R)\PAR{1-e^{-\bar a R}}>{\alpha+1\over \alpha -1}\bar a.
\]
Thanks to \eqref{eq:phi-} and \eqref{eq:phi+}, the assumptions of Lemma~\ref{lem:taR} 
are satisfied for all $\lambda\in (0,\lambda_c)$,  where
\[
\lambda_c:=\inf\BRA{\lambda>0:(\alpha +1)\lambda\geq b(R)
\PAR{1-e^{-\PAR{\bar a+(\alpha-1)\lambda} R}}-{\alpha+1\over\alpha -1}\bar a}.
\]
Indeed, from  \eqref{eq:phi-} and \eqref{eq:phi+}, condition  
$\varphi_\PAR{R,+1}(\beta+\lambda) \varphi_\PAR{R,-1}(\lambda-\beta)  <1$ 
holds as soon as 
\begin{eqnarray*}
&&b(R)\PAR{\bar a+(\alpha-1)\lambda e^{-\PAR{\bar a+(\alpha -1)\lambda }R}}
<\PAR{\bar a+(\alpha-1)\lambda}\PAR{b(R)-(\alpha+1)\lambda}\\
 &\Leftrightarrow & b(R)(\alpha -1)\lambda e^{-\PAR{\bar a+(\alpha -1)\lambda }R}
 <(\alpha-1)\lambda b(R)-\PAR{\bar a+(\alpha-1)\lambda}(\alpha+1)\lambda\\
&\Leftrightarrow &(\alpha+1)\lambda<
b(R)\PAR{1-e^{-\PAR{\bar a+(\alpha -1)\lambda }R}}-{\alpha+1\over \alpha -1}\bar a.
\end{eqnarray*}
Since the previous inequality is satisfied for $\lambda=0$, by continuity we have 
$\lambda_c>0$;  the function of $\lambda$ on the r.h.s. being strictly concave,  we also have
$\lambda_c<\infty$.


Finally,  notice that by Lemma \ref{lem:EvolTx}  and Holder's inequality,   for any $q>1$, 
 \[
{\cal E}_R(\gamma)\leq (1-p_R)^{1-1/q}\varphi_\PAR{0,+1}^{1/q}(q\gamma),
 \] 
with $p_R\in (0,1)$ a quantity as in \eqref{eq:pR}. Taking $q=q(\gamma)=\gamma^{-1}$, this in turn yields, for each fixed $R>0$,  
$\limsup\limits_{\gamma\to 0}  {\cal E} _R(\gamma)\leq  (1-p_R)<1.$
Therefore,  there exists $\lambda_c' \in (0,\lambda_c)$ small enough such that 
${\cal E}_R((\alpha+1)\lambda)<1$ for all $\lambda \in (0,\lambda_c')$ .

\subsection{The coupling time for the reflected process}
Let us consider two initial data $(x,v)$ and $(\tilde x,\tilde v)$, with $x\geq \tilde x$. 
The coalescent time $T_{*}(x,\tilde x)$ of a path $\PAR{X,V}$ starting from $(x,v)$ 
and a path $\PAR{\tilde X,\tilde V}$ starting from $(\tilde x,\tilde v)$ is equal to the 
first crossing time $T_c(x,\tilde x)$ of both paths plus the time  spent to stick them 
using the coupling described in Section \ref{sec:stick}. Consequently, the 
coupling time is stochastically smaller than the hitting time $Z_{(x,v)}$ of  the origin of the 
upper path $\PAR{X,V}$, plus some remainder term. 

For any $(R,\lambda,\beta)$ satisfying assumptions of Proposition \ref{prop:stick}, the Laplace 
transform of the coupling time $T_*(x,\tilde x)$ is bounded by
\[
\dE\SBRA{e^{\lambda T_*(x,\tilde x)}}\leq   \frac{
 \varphi_\PAR{0,+1}(\lambda)\varphi_\PAR{R,-1}(\lambda)}{1- {\cal E} _R(\lambda +\beta) }
 \dE\SBRA{e^{\lambda T_c(x,\tilde x)}e^{\beta X_c \ind_{X_c> R}}}.
\]
where the first crossing time $T_c(x,\tilde x)$ is smaller than $Z_{(x,v)}$ and the first 
crossing point $X_c$ is bounded from above by $\frac{1}{2}\PAR{Z_\PAR{x,v}+\tilde x-x}$.
Consequently,
\[
\dE\SBRA{e^{\lambda T_*(x,\tilde x)}}\leq   \frac{
 \varphi_\PAR{0,+1}(\lambda)\varphi_\PAR{R,-1}(\lambda)}{1- {\cal E} _R(\lambda +\beta)}
 e^{\beta(\tilde x-x)/2}
 \dE\SBRA{e^{(\lambda+\beta/2) Z_\PAR{x,v}}}.
\]
Using now Proposition \ref{prop:zero} , for $0<\lambda<\beta$ satisfying conditions of 
Proposition~\ref{prop:stick} with $\lambda+\beta/2<\frac{1}{2}(\sqrt{b(M_c)}-\sqrt{a(M_c)})^2$, 
we finally get
\begin{equation}\label{eq:estimT*}
\dE\SBRA{e^{\lambda T_*(x,\tilde x)}}\leq   C
\frac{\varphi_\PAR{0,+1}(\lambda)\varphi_\PAR{R,-1}(\lambda)}{1- {\cal E} _R(\lambda +\beta)}
 e^{\frac{\beta(\tilde x-x)}{2}}
e^{\frac{x(b(M_c)-a(M_c))}{2}}
\end{equation}
with $C$ is given in Proposition~\ref{prop:zero} and 
\[
M_c=\sup\BRA{M>0\ :\ 
\sqrt{\frac{b(M)}{a(M)}}e^{ M\PAR{\sqrt{b(M)}-\sqrt{a(M)}}^2}\PAR{1-e^{-A(M)}}<1}.
\]
\noindent This concludes the proof of Theorem~\ref{th:conv-reflection}.

\section{The unreflected process}\label{sec:non-reflected}

Let us construct a coalescent coupling of two unreflected processes starting 
from $(y,w)$ and $(\tilde y, \tilde w)$ respectively. 
For a given time $t_0>0$, the coupling algorithm is the following: 
\begin{enumerate}
 \item Define $(x,v)=(|y|,w\, \text{sgn}(y))$ and 
 $(\tilde x, \tilde v)=(|\tilde y|,\tilde w\, \text{sgn}(\tilde y))$.
 \item Couple two reflected processes  starting at $(x,v)$ and $(\tilde x,\tilde v)$ as in Section \ref{sec:coalescent-refl}. 
 \item Let them evolve until their common   hitting time of  $0$.
 \item Construct until that time   the two associated unreflected processes starting at $(y,w)$ 
 and $(\tilde y, \tilde w)$ as explained in Remark~\ref{rem:ru}. The algorithm stops if,  when  at the origin, the two copies have the same velocities. Otherwise, go to Step 5.
  \item Try to couple the unreflected processes starting from $(0,+1)$ and $(0,-1)$ before 
 a fixed time $t_0$.
 \item In case of failure, return to step 1  for two initial conditions in $[-t_0,t_0]\times\BRA{-1,+1}$. 
\end{enumerate}

The only remaining  task is to analyze Step 5 of this algorithm. 
To that end, one has to study the law of $(Y_t,W_t)$ when $Y_0=0$ and 
$W_0=\pm1$. Let us denote by ${(T_n)}_{n\geq 0}$ (with $T_0=0$) 
the successive jump times of the unreflected process. The variable $S_n=T_n-T_{n-1}$ 
stands for the $n^{th}$ inter-jump time.  In order to lighten the computation, we restrict ourselves to the law  after 1 or 2 jumps. 

\begin{rem}[Jump times of the unreflected process]
We can explicitly compute  the law of the jump-times of the 
unreflected process. For $y\in\dR$, set $A(y)=A(\ABS{y})$ and $B(y)=B(\ABS{y})$ 
where $A$ and $B$ were defined on $\dR_+$ in Lemma~\ref{lem:law-time}.
For $y>0$, the law of the first jump time starting from $(y,-1)$ 
has the density $f_\PAR{y,-1}$ given by 
\[
f_\PAR{y,-1}(t)=
\begin{cases}
 a(y-t)e^{-(A(y)-A(y-t))} &\text{if }t<y,\\
 e^{-A(y)}  b(t-y)e^{-B(t-y)}&\text{if }t\geq y.
\end{cases}
\]
Moreover, the survival function $\bar F_\PAR{y,+1}(t):=\dP_\PAR{y,+1}(T_1>t)$ is given 
for $y\geq 0$ by
\[
\bar F_\PAR{y,+1}(t)=e^{-(B(y+t)-B(y))}, 
\]
and  for $y<0$ by 
\[
\bar F_\PAR{y,+1}(t)=
\begin{cases}
e^{-(A(y)-A(y+t))}  &\text{if }y+t<0, \\
e^{-A(y)} e^{-B(t+y)} &\text{if }y+t\geq 0.  
\end{cases}
\]
\end{rem}
\noindent For any bounded measurable function $g$ on $\dR\times \BRA{-1,+1}$, one  then has 
\begin{align*}
\dE_\PAR{0,-1}\SBRA{g(Y_t,W_t)\ind_\BRA{T_1<t,T_2>t}} 
&=\dE_\PAR{0,-1}\SBRA{g(t-2S_1,+1)\ind_\BRA{S_1<t,S_1+S_2>t}} \\
&=\int_0^t\! g(t-2s,+1) f_{(0,-1)}(-s)\bar F_\PAR{-s,+1}(t-s)\,ds\\
&=\int_{-t}^t\! g(u,+1) h_{-1}(u)\,du,  
\end{align*}
where 
\[
h_{-1}(u)=\frac{1}{2}f_{(0,+1)}\PAR{\frac{t-u}{2}}\bar F_\PAR{\frac{t-u}{2},+1}\PAR{\frac{t+u}{2}}.
\]
Similarly,
\begin{align*}
&\dE_\PAR{0,+1}\SBRA{g(Y_t,W_t)\ind_{T_2<t,T_3>t}} 
=\dE_\PAR{0,+1}\SBRA{g(t-2S_2,+1)\ind_{S_1+S_2<t,T_3>t}} \\
&\quad\quad=\int_0^t\!\int_0^{t-s_2}\! g(t-2s_2,+1) f_{(0,+1)}(s_1)f_{(s_1,-1)}(s_2)
\bar F_\PAR{s_1-s_2,+1}(t-s_1-s_2)\,ds_1\,ds_2\\
&\quad\quad=\int_{-t}^t\! g(u,+1) h_1(u)\,du,  
\end{align*}
where 
\[
h_1(u)=\int_0^{\frac{t+u}{2}}\! \frac{1}{2}f_{(0,+1)}(s_1)f_{(s_1,-1)}\PAR{\frac{t-u}{2}} 
\bar F_\PAR{s_1-\frac{t-u}{2},+1}\PAR{\frac{t+u}{2}-s_1}ds_1.
\]
Since $\cL((Y_t,W_t)\vert Y_0=0,W_0=w)=\cL((-Y_t,W_t)\vert  Y_0=0,W_0=-w)$,  two copies as in step 5 of the algorithm  couple before time $t\geq 0$ with probability larger than
\[
\varepsilon_t=2\int_{-t}^t\! h_{-1}(u)\wedge h_1(u)\,du>0. 
\]

\begin{rem}[Explicit lower bound]
 A lower bound of $\varepsilon$ can be derived from the fact that $y\mapsto a(y)$ 
 and $y\mapsto b(y)$ respectively belong to $[a(t),a(0)]$ and $[b(0),b(t)]$ on the interval $[-t,t]$. 
\end{rem}

Let us now control the total  duration of the algorithm. 
Notice that  the estimates on the reflected process in  Section \ref{sec:coalescent-refl}  do no longer depend on the initial conditions, after the first crossing time of the reflected copies in  the compact set $[-R,R]\times\BRA{-1,+1}$.  This implies that, after the first iteration of Step 2 in the above algorithm, the duration of each step can be controlled independently of the initial data, and of the previous steps. 

Moreover, the algorithm succeeds after 
 at  most a random number of iterations with geometric law of parameter
$\varepsilon_{t_0}$.  Since the duration of each step in the algorithm has a 
finite exponential moment, this is thus  true for the coupling time as well. The upper 
bound of Theorem~\ref{th:no-reflection} can then be deduced. As a conclusion the bounds in Theorems~\ref{th:no-reflection}
and \ref{th:conv-reflection} depends in the same way on initial data but the rate of convergence is 
smaller for the unreflected process. 

\section{Diffusive scaling}\label{sec:diffus}
We finally prove Theorem~\ref{prop:difulimit}. Omitting  for a moment the sub and superscripts for notational 
simplicity, and writing 
\[
j_t:=  W_t+\kappa'(Y_{t}) -2 \PAR{a(Y_t)\ind_\BRA{Y_t W_t\leq 0}+b(Y_{t})\ind_\BRA{Y_t W_t>0}} \kappa(Y_{t}) W_t,
\]
$J_t: =\int_0^t j_s ds$ and $\hat{Y}_t:=Y_t+ \kappa(Y_{t}) W_t$ for a given positive
function $\kappa $ of class ${\cal C}^1$, we see by Dynkin's theorem that the processes 
\[
M_t:= \hat{Y}_t-J_t =Y_t +\kappa(Y_{t}) W_t -J_t
\]
and 
\[
N_{t}:=\hat{Y}_t^2- 2\int_{0}^{t}\kappa(Y_{s})ds 
-  2\int_0 ^t Y_s j_s ds-2\int_{0}^t\kappa'(Y_{s})\kappa(Y_{s})W_{s}ds
\]
are local martingales with respect to the  filtration generated by $(Y_t,W_t)$. 

In fact, using $f(y,w)=y+\kappa(y)w$ and $g(y,w)=(y+\kappa(y)w)^2$, since $w^2=1$, we have
\begin{align*}
Lf(y,w)&=w+\kappa'(y)-2\PAR{a(y)\ind_\BRA{yw\leq 0}+b(y)\ind_\BRA{yw>0}}\kappa(y)w,\\
Lg(y,w)&=2w(1+\kappa'(y)w)(y+\kappa(y)w)-
4\PAR{a(y)\ind_\BRA{yw\leq 0}+b(y)\ind_\BRA{yw>0}}y\kappa(y)w\\
&=2\kappa(y)+2y\PAR{w+\kappa'(y)
-2\PAR{a(y)\ind_\BRA{yw\leq 0}+b(y)\ind_\BRA{yw>0}}k(y)w}+2\kappa'(y)\kappa(y)w.
\end{align*}
Integrating by parts we then get that 
\begin{align*}
M_t^2&= \hat{Y}_t^2-2\hat{Y}_tJ_{t}+J_{t}^2\\
&=\hat{Y}_t^2
-2\int_{0}^t({Y}_s+\kappa(Y_{s})W_{s})j_{s}ds
-2\int_{0}^tJ_{s}d\hat Y_{s}+2\int_{0}^tJ_{s}j_{s}ds\\
&=N_t+2\int_{0}^{t}\kappa(Y_{s})ds 
-2\int_{0}^t\kappa(Y_{s})W_{s}j_{s}ds
+2\int_{0}^t\kappa'(Y_{s})\kappa(Y_{s})W_{s}ds\\
&\quad\quad\quad -2\int_{0}^tJ_{s}d\hat Y_{s} +2\int_{0}^tJ_{s}j_{s}ds \\
&=N_t+2\int_{0}^t\kappa(Y_{s})ds 
-2\int_0^t \kappa(Y_{s})W_s [j_s-\kappa'(Y_{s})] ds 
- 2\int_0^t J_{s} dM_s.
\end{align*} 
Thus, noting that  
\[
j_s=\kappa'(Y_{s})+\mathrm{sgn}(Y_s) \left( \left[2a(Y_{s})\kappa(Y_{s}) -1\right] 
+2 \ind_\BRA{Y_sW_s>0}\left
[1-\kappa(Y_{s}) (a(Y_{s})+b(Y_{s}))\right]\right)
\]
we see for $\kappa(Y_{s})=\PAR{a(Y_{s})+b(Y_{s})}^{-1}$ that 
\begin{align*}
M_t&=Y_t- \left[-  \int_0 ^t \left[\frac{a'(Y_{s})+b'(Y_{s})}{\left(a(Y_{s})+b(Y_{s})\right)^2}  
+\mathrm{sgn}(Y_s) \left(\frac{b(Y_{s})-a(Y_{s})}{a(Y_{s})+b(Y_{s})}\right)\right]ds  
-\frac{W_t}{a(Y_{t})+b(Y_{t})}\right] \, , 
\, \\
M_t^2&- 2\int_{0}^t\left[\frac{1}{a(Y_{s})+b(Y_{s})} +W_s\, 
\mathrm{sgn}(Y_s)  \frac{b(Y_{s})-a(Y_{s})}{(a(Y_{s})+b(Y_{s}))^2}  \right] ds 
\end{align*}
are local martingales.

The function $\tau_{\N}$ of the statement  is well defined by the Cauchy-Lipschitz 
Theorem, thanks to the assumptions on the coefficients and the fact that $Y_t$ has 
Lipschitz trajectories. Moreover,  $\tau_{\N}$  is strictly increasing, the coefficients 
$a$ and $b$ being positive functions. Recalling the dependence on $N$ of the coefficients, 
and defining for each $N\in \dN$,
\begin{align*}
\beta^{(\N)}_t:=-&\frac{1}{2}\int_{0}^t\!\SBRA{
\frac{a_N'(\xi_s^{(\N)})+b_N'(\xi_s^{(\N)})}{a_{N}(\xi_s^{(\N)})+b_N(\xi_s^{(\N)})}
+ \mathrm{sgn}(\xi_s^{(\N)}) \PAR{b_N(\xi_s^{(\N)})-a_N(\xi_s^{(\N)})}}ds \\
 &- \frac{W^{(\N)}_{\tau_{\N}(t)}}{a_\N(\xi_t^{(\N)})+b_N(\xi_t^{(\N)})}
\end{align*}
and 
\[
\alpha^{(\N)}_t:= t+\frac{1}{2}
\int_0^t \! W^{(\N)}_{\tau_{\N}(s)}\, \mathrm{sgn}(\xi_s^{(\N)}) 
\frac{b_N(\xi_s^{(\N)})-a_N(\xi_s^{(\N)})}{a_{N}(\xi_s^{(\N)})+b_N(\xi_s^{(\N)})} ds, 
\]
we see from the previous and Doob's optional stopping theorem that the processes
 \[
\PAR{\xi^{(\N)}_\cdot- \beta^{(\N)}_t}_{t\geq 0}
\quad\text{ and }\quad
\PAR{(\xi^{(\N)}_t- \beta^{(\N)}_t )^2- \alpha^{(\N)}_t}_{t\geq 0}
\]
are local martingales with respect to the filtration of the process 
$(\xi^{(\N)}_t, W^{(\N)}_{\tau_{\N}(t)})_{t\geq 0}$. Defining for each $R>0$ the stopping time  
$\sigma^N_R =\inf\{t\geq 0 : |\xi^{(\N)}_t| \geq R \mbox{ of } |\xi^{(\N)}_{t-}| \geq R\},$
we notice that the hypotheses imply that
\[
\dP\left(\sup_{t\leq \sigma^N_R} \left| \alpha^{(\N)}_t - t\right| \geq \varepsilon \right) 
+ \dP\left(\sup_{t\leq \sigma^N_R} \left| \beta^{(\N)}_t+ 
\int_0^t\!\mathrm{sgn}(\xi_s) c_1(\xi_s)+c_2(\xi_s) \,ds\right|\geq \varepsilon \right)  
\xrightarrow[N\to\infty]{} 0
\]
for all $\varepsilon >0$ and  
\[
\dE\left(\sup_{t\leq \sigma^N_R} \left| \xi^{(\N)}_t - \xi^{(\N)}_{t-} \right|^2 \right) 
+ \dE\left(\sup_{t\leq \sigma^N_R} \left| \beta^{(\N)}_t- \beta^{(\N)}_{t-}\right|^2 \right)
\xrightarrow[N\to\infty]{} 0.
\]
The processes $\PAR{\xi^{(\N)}_t, \alpha^{(\N)}_t}_{t\geq 0}$ and 
$\PAR{\beta^{(\N)}_t}_{t\geq 0}$ thus satisfy the hypotheses of  
Theorem~4.1  in \cite[p. 354]{ethier} (in the respective roles of the 
processes~$X_n(\cdot),A_n(\cdot)$ 
and $B_n(\cdot)$ therein), which ensures  that  $\cL((\xi^{(\N)}_t,t\geq 0))$ 
converges weakly to the unique solution of the martingale problem with 
generator given for $f\in \mathcal{C}_c^{\infty}(\dR)$ by 
\[
Gf(x):= \frac{1}{2}f''(x) -\,(\mathrm{sgn}(x) c_1(x)+c_2(x))f'(x)
\]
 and initial law $\cL(\xi_0)$.

\bigskip

{\bf Acknowledgements.} The second and third authors  thank support of Agence Nationale de la Recherche project
PIECE 12-JS01-0006-0. The first author thanks the support of Conicyt-Chile Basal Center for Mathematical Modeling,  Fondecyt Project 1150570 and Nucleo Milenio NC120062.

\addcontentsline{toc}{section}{\refname}%
\bibliography{FGM}

\bigskip

\begin{flushright}\texttt{Compiled \today.}\end{flushright}

{\footnotesize %
  \noindent Joaquin \textsc{Fontbona},
 e-mail: \texttt{fontbona(AT)dim.uchile.cl}

 \medskip

 \noindent\textsc{Center for Mathematical Modeling,  UMI 2807 UChile-CNRS, Universidad de Chile,
 Casilla 170-3, Correo 3, Santiago, Chile.}

 \bigskip

 \noindent H\'el\`ene \textsc{Gu\'erin},
e-mail: \texttt{helene.guerin(AT)univ-rennes1.fr}

 \medskip

  \noindent\textsc{Institut de Recherche Math\'ematique de
    Rennes (UMR CNRS 6625) \\ Universit\'e de Rennes 1, Campus de Beaulieu, F-35042
    Rennes \textsc{Cedex}, France.}

 \bigskip

 \noindent Florent \textsc{Malrieu},
 e-mail: \texttt{florent.malrieu(AT)univ-tours.fr}

 \medskip

\noindent\textsc{Laboratoire de Mat\'ematiques et Physique Th\'eorique (UMR
CNRS 7350), F\'ed\'eration Denis Poisson (FR CNRS
2964), Universit\'e Fran\c cois-Rabelais, Parc de Grandmont,
37200 Tours, France.}

}

\end{document}